\documentclass[a4paper, 12pt]{amsart}

\usepackage[notcite, notref]{showkeys}

\usepackage{amsthm}
\usepackage[english]{babel}                             
\usepackage[latin1]{inputenc}
\usepackage{comment}
\usepackage{ae}
\usepackage[leqno]{amsmath}  

\usepackage{amsthm}                    

\usepackage[dotinlabels]{titletoc}     

\usepackage{amssymb,latexsym,verbatim} 
\usepackage{amstext}
\usepackage{amsfonts}
\usepackage{amsbsy}
\usepackage{amsopn}
\usepackage{enumerate}
\usepackage{txfonts}
\usepackage{amsmath, amsthm, amssymb}

\usepackage{color}
\usepackage[usenames,dvipsnames,svgnames,table]{xcolor}

\allowdisplaybreaks
\newtheorem{theorem}{Theorem}
\newtheorem{definition}[theorem]{Definition}
\newtheorem{lemma}[theorem]{Lemma}

\newtheorem{proposition}[theorem]{Proposition}

\numberwithin{equation}{section}
\numberwithin{theorem}{section}

\def\N{\mathbb{N}}
\def\R{\mathbb{R}}

\renewcommand{\epsilon}{\varepsilon}
\renewcommand{\rho}{\varrho}
\renewcommand{\widetilde}{\tilde}
\newcommand{\qtilde}{\tilde q}
\newcommand{\widetildelambda}{\widetilde{\lambda}}
\DeclareMathOperator{\supp}{supp}

\begin{document}
\title[Subquadratic parabolic systems with non-standard growth]{Very weak solutions of subquadratic parabolic systems with non-standard $p(x,t)$-growth}
\author[Q. Li]{Qifan Li}
\address{Qifan Li\\
Department of Mathematics, School of Sciences, Wuhan University of Technology\\
430070, 122 Luoshi Road, Wuhan, Hubei, P. R. China}
\email{qifan\_li@yahoo.com, qifan\_li@whut.edu.cn}
\maketitle
\begin{abstract}
The aim of this paper is to establish a higher integrability result for very weak solutions of certain  parabolic systems whose model is the parabolic
$p(x,t)$-Laplacian system. Under assumptions on the exponent function $p:\Omega_T=\Omega\times (0,T)\to\left(\frac{2n}{n+2},2\right]$, it is shown that any very weak solution $u:\Omega_T\rightarrow\mathbb{R}^N$ with $|Du|^{p(\cdot)(1-\varepsilon)}\in L^1(\Omega_T)$ belongs to the natural energy spaces, i.e. $|Du|^{p(\cdot)}\in L^1_{\operatorname{loc}}(\Omega_T)$, provided $\epsilon>0$ is small enough.
This extends the main result of [V. B\"ogelein and Q. Li, Nonlinear Anal., 98 (2014), pp. 190-225] to the subquadratic case.
\end{abstract}
\section{Introduction}
The reverse H\"older inequality for the weak and very weak solutions of parabolic systems was first studied by
Kinnunen and Lewis \cite{KL1,KL2}; see also \cite{bo,b,B} for the case of higher order systems. Later on Zhikov and Pastukhova \cite{ZP} and independently B\"ogelein and Duzaar \cite{BD} proved the higher integrability of weak solutions to parabolic systems with non-standard $p(x,t)$-growth whose model is the parabolic $p(x,t)$-Laplacian system:
$$\frac{\partial u}{\partial t} - \operatorname{div} \bigl(|Du|^{p(x,t)-2}Du\bigr)=\operatorname{div} \bigl(|F|^{p(x,t)-2}F\bigr),$$
where $p(x,t)$ is logarithmically continuous. Recently,
B\"ogelein and the author \cite{bl} studied the very weak solutions to this kind of parabolic systems with superquadratic growth. This problem was suggested as an open problem in the overview article \cite{HHLN}. In this paper, we extend the higher integrability result of \cite{bl} to the subquadratic case.

In the subquadratic case, the lower bound on $p(x,t)$, i.e. $p(x,t)>\frac{2n}{n+2}$, is a typical assumption in the regularity theory for non-linear parabolic equations, cf. \cite{Di}. Our proof is in the spirit of \cite{bl,KL2} and we will work with a "non-standard version" of the intrinsic geometry introduced by B\"ogelein and Duzaar \cite{BD}.
However, the proof in \cite{bl} strongly depended on the assumption that $p(x,t)\geq2$.
The major difficulty in our subquadratic case is that the parabolic cylinder $$Q_\rho^{(\lambda)}(z_0)=B_\rho(x_0)\times\left(t_0-\lambda^{\frac{2-p(z_0)}{p(z_0)}}\rho^2,t_0+\lambda^{\frac{2-p(z_0)}{p(z_0)}}\rho^2\right)$$ studied in \cite{BD,bl}
cannot directly apply to the subquadratic case, since $Q_\rho^{(\lambda)}(z_0)\subset \Omega_T$ may fail when $p(z_0)<2$ and $\lambda$ large enough.
So we use another scaled parabolic cylinder $\widetilde Q_\rho^{(\lambda)}(z_0)$:
$$\widetilde Q_\rho^{(\lambda)}(z_0)=B_{\lambda^{\frac{p(z_0)-2}{2p(z_0)}}\rho}(x_0)\times\left(t_0-\rho^2,t_0+\rho^2\right)$$
instead to deal with our subquadratic case. We also remark that as pointed out by Kopaliani \cite{Ko}, the strong maximal functions are not bounded in $L^{p(\cdot)}$ unless $p(\cdot)\equiv \text{constant}$. As a consequence, we have to estimate the strong maximal functions in the usual Lebesgue spaces in the proof.

This paper is organized as follows. We state the main result in \S 2. In \S 3, we provide some preliminary material, while in \S 4 we construct the testing function for the parabolic system. In \S 5 we give the proof of the Caccioppoli inequality. \S 6 is devoted to the proof of the reverse H\"older inequality under an additional assumption. Finally, in \S 7, we follow with the arguments as \cite[\S 9]{bl} and \cite[\S 7]{BD} to obtain the higher integrability of very weak solutions. Since the argument is standard, we only sketch the proof in this section.

\section{Statement of the Main Result}
In the following $\Omega$ will denote a bounded domain in $\mathbb{R}^n$ with $n\geq2$ and $\Omega_T=\Omega\times(0,T)\subset \mathbb{R}^{n+1}$, $T>0$
will be the space-time cylinder. We denote by $Du$ the differentiation with respect to the space variables, while $\partial_tu$ stands for the time derivative.
Points in $\mathbb{R}^{n+1}$ will be denoted by $z=(x,t)$. We shall use the parabolic cylinders of the form $Q_{\rho}(z_0)=B_{\rho}(x_0)\times(t_0-\rho^2,t_0+\rho^2)$ where $B_{\rho}(x_0)=\{x\in \mathbb{R}^n:|x-x_0|\leq\rho\}$.

We consider degenerate parabolic systems of the following type
\begin{equation}\label{3}\partial_t u - \operatorname{div} A(z,Du)=B(z,Du),\end{equation} where the vector fields $A,B:\Omega_T\times\mathbb{R}^{nN}\rightarrow\mathbb{R}^{nN}$ satisfy the following non-standard $p(z)$-growth and ellipticity conditions:
\begin{equation}\label{4}\begin{split}&|A(z,\xi)|\leq L(1+|\xi|+|F|)^{p(z)-1}\\&|B(z,\xi)|\leq L(1+|\xi|+|F|)^{p(z)-1}\\&\langle A(z,\xi),\xi\rangle\geq \nu|\xi|^{p(z)}-|F|^{p(z)}
\end{split}\end{equation}for any $z\in \Omega_T$, $\xi\in\mathbb{R}^{nN}$. Here, $F\colon\Omega_T\to\mathbb{R}^{nN}$ with $|F|^{p(\cdot)}\in L^1(\Omega_T)$ and $0<\nu<L$ are fixed structural parameters.
For the exponent function $p:\ \Omega_T\rightarrow(\frac{2n}{n+2},2]$ we assume that it is continuous with a moduls of continuity $\omega:\ \Omega_T\rightarrow[0,1]$. More precisely, we assume that
\begin{equation}\label{2}\frac{2n}{n+2}<\gamma_1\leq p(z)\leq 2\ \ \ \text{and}\ \ \ |p(z_1)-p(z_2)|\leq\omega(d_P(z_1,z_2)),
\end{equation}
holds for any $z,\ z_1,\ z_2\in\Omega_T$ and some $\frac{2n}{n+2}<\gamma_1\leq2$. Since our estimates are of local nature, it is not restrictive to assume a lower bound for $p(\cdot)$.
As usual, the parabolic distance $d_P$ is given by
$$d_P(z_1,z_2):=\max\left\{|x_1-x_2|,\sqrt{|t_1-t_2|}\right\},\quad\mbox{for $z_1=(x_1,t_1),z_2=(x_2,t_2)\in\mathbb{R}^{n+1}$.}$$
The modulus of continuity $\omega$ is assumed to be a concave, non-decreasing function satisfying the following weak logarithmic
continuity condition:
\begin{equation}\label{1}\sup_{0\leq\rho\leq1}\omega(\rho)\log\left(\tfrac{1}{\rho}\right)<L<+\infty.\end{equation}
The spaces $L^p(\Omega,R^N)$ and $W^{1,p}(\Omega,R^N)$ are the usual Lebesgue and Sobolev spaces. Moreover, for a variable exponent $p(\cdot)$, we denote by $L^{p(\cdot)}(\Omega_T,\mathbb{R}^k)$, $k\in\N$, the variable exponent Lebesgue space
$$
	L^{p(\cdot)}(\Omega_T,\mathbb{R}^k):=\left\{v\in L^1(\Omega_T,\mathbb{R}^{k}):\ \int_{\Omega_T}|v|^{p(\cdot)}dz<\infty\right\}.
$$
Throughout the paper, unless otherwise stated, we denote by $c$ the constant depends only on $n$, $N$, $L$, $\nu$ and $\gamma_1$.

\begin{definition}\upshape
Let $\epsilon\in (0,1)$. We say that $u\in L^2(\Omega_T,\R^N)$ is a {\it very weak solution to the parabolic system \eqref{3} with deficit $\epsilon$} if and only if
$$
u\in L^{p(\cdot)(1-\epsilon)}(\Omega_T,\R^{N})
\quad\mbox{and}\quad
Du\in L^{p(\cdot)(1-\epsilon)}(\Omega_T,\R^{Nn})
$$
and
\begin{equation}\label{5}
\int_{\Omega_T}u\cdot \partial_t \varphi - \langle A(z,Du), D\varphi\rangle\, dz=\int_{\Omega_T}B(z,Du)\cdot\varphi \,dz\end{equation}
holds, whenever
$\varphi\in C_0^{\infty}(\Omega_T,\mathbb{R}^N)$.
\end{definition}

The following theorem is our main result.
\begin{theorem}\label{main theorem}Let $p:\ \Omega_T\rightarrow[\gamma_1,2]$ satisfies (\ref{2}) and (\ref{1}). Then there exists $\varepsilon_0=\varepsilon_0(n,N,\nu,L,\gamma_1)>0$ such that the following holds: Whenever $u\in L^2(\Omega_T,\mathbb{R}^N)\cap L^{p(z)(1-\varepsilon)}(\Omega_T,\mathbb{R}^N)$ and $|Du|^{p(z)(1-\varepsilon)}\in L^1(\Omega_T)$ with some $\varepsilon\in(0,\epsilon_0]$ is a very weak solution of the parabolic system (\ref{5}) under the assumptions (\ref{4}) and $F\in L^{p(z)}(\Omega_T,\mathbb{R}^{nN})$, then we have
$$|Du|^{p(z)}\in L_{\operatorname{loc}}^1(\Omega_T).$$
Moreover, for $M\geq1$ there exists a radius $r_0=r_0(n,N,\nu,L,\gamma_1,M)$ such that there holds: If
\begin{equation}\label{6}\int_{\Omega_T}\left(|u|+|Du|+|F|+1\right)^{p(\cdot)(1-\varepsilon)}dz\leq M\end{equation}and $\varepsilon\in(0,\varepsilon_0]$, then for any parabolic cylinder $Q_{2r}(\mathfrak z_0)\subseteq\Omega_T$ with $r\in(0,r_0]$ there holds
\begin{equation}\label{7}\begin{split}\fint_{Q_{r}(\mathfrak z_0)}|Du|^{p(z)}dz\leq & c\left(\fint_{Q_{2r}(\mathfrak z_0)}(|Du|+|F|+1)^{p(z)(1-\varepsilon)}dz\right)^{1+ 2\epsilon d}\\&+c\fint_{Q_{2r}(\mathfrak z_0)}\left(|F|+1\right)^{p(z)}dz,
\end{split}\end{equation}where
$$d\equiv d(p_0)=\frac{2p_0}{p_0(n+2)-2n}$$and $p_0=p(\mathfrak z_0)$.
\end{theorem}

\section{Preliminary material and notation}\label{preliminary}
For a point $z_0=(x_0,t_0)\in \mathbb{R}^{n+1}$ and parameters $\rho>0$, $\lambda>0$, we define the scaled cylinder $Q^{(\lambda)}_{\rho}(z_0)$ by $Q^{(\lambda)}_{\rho}(z_0):=B_{\rho}^{(\lambda)}(x_0)\times\Lambda_{\rho}(z_0)$ where $ \Lambda_{\rho}(z_0):=(t_0-\rho^2,t_0+\rho^2),$ $ B_{\rho}^{(\lambda)}(x_0):=\{x\in \mathbb{R}^n:|x-x_0|\leq \lambda^{(p_0-2)/(2p_0)}\rho\}$ and $p_0=p(z_0)$. For $\alpha>0$, we write $\alpha Q^{(\lambda)}_{\rho}(z_0)$ for the scaled cylinder $Q^{(\lambda)}_{\alpha\rho}(z_0)$.
Moreover, for a function $f\in L^1(\R^{n+1}, \R^k)$, $k\in\N$ we define its strong maximal function by
$$
	M(f)(z)
	:=
	\sup\left\{\fint_Q|f|d\tilde{z}:\ \ z\in Q,\ \  Q\ \mathrm{is}\ \mathrm{parabolic}\ \mathrm{cylinder}\right\}.
$$
Here, by parabolic cylinder we mean that $Q$ is a cylinder of the form $B\times\Lambda$ where $B$ is a ball in $\R^n$ and $\Lambda\subset\R$ is an interval.
To simplify the notations, we write $f_{G}$ instead of $\fint_Gfdz$ for any subset $G\subset\mathbb{R}^{n+1}$.
We will use the following iteration lemma, which is a standard tool and can be found in \cite{G}.

\begin{lemma}\label{iteration lemma}Let $0<\theta<1$, $C_1,C_2\geq0$ and $\beta>0$. Then there exists a constant $A=A(\theta,\beta)>0$ such that there holds: For any non-negative bounded function $\phi(t)$ satisfying
$$\phi(t)\leq\theta\phi(s)+C_1(s-t)^{-\beta}+C_2\ \ \ \text{for\ \ all}\ \ \ \ 0<r\leq t<s\leq\rho,$$
we have $$\phi(t)\leq A\left[C_1(\rho-r)^{-\beta}+C_2\right].$$
\end{lemma}
Next, we state Gagliardo-Nirenberg's inequality in a form which
shall be convenient for our purposes later.
\begin{lemma}\label{Gagliardo-Nirenberg}Let $B_{\rho}(x_0)\subset\mathbb{R}^n$ with $0<\rho\leq1$, $1\leq\sigma,q,r<+\infty$ and $\theta\in(0,1)$ such that $-n/\sigma\leq\theta(1-n/q)-(1-\theta)n/r$. Then there exists a constant $c=c(\sigma,n)$ such that for any $u\in W^{1,q}(B_{\rho}(x_0))$ there holds:
$$\fint_{B_{\rho}(x_0)}\left|\frac{u}{\rho}\right|^{\sigma}dx\leq c\left(\fint_{B_{\rho}(x_0)}\left|\frac{u}{\rho}\right|^q+|Du|^qdx\right)^{\theta \sigma/q}\left(\fint_{B_{\rho}(x_0)}\left|\frac{u}{\rho}\right|^rdx\right)^{(1-\theta)\sigma/r}.$$
\end{lemma}
We now reformulate the parabolic system (\ref{5}) in its Steklov form as follows:
\begin{equation}\label{11}\int_{\Omega}\partial_t[u]_h(\cdot,t) \varphi + \langle [A(z,Du)]_h,D\varphi\rangle(\cdot,t) dx=-\int_{\Omega}\langle[B(z,Du)]_h,\varphi\rangle (\cdot,t) dx\end{equation}for all
$\varphi\in C_0^{\infty}(\Omega,\mathbb{R}^N)$ and a.e. $t\in(0,T)$. For the proof of (\ref{11}), we refer the reader to \cite[Chapter 8.2]{bo}.

Since we have to derive estimates on intersections of parabolic cylinders, we will formulate Sobolev-Poincar\'e type estimates for very general types of sets.
The first one can be deduced from \cite[Lemma 4.1, Lemma 4.2]{B}, \cite[Lemma 5.1]{BD} and the second Lemma can be easily derived from Lemma \ref{sobolev1}, \cite[Corollary 5.2]{bl}. The proof of these two Lemmas will be omitted.
\begin{lemma}\label{sobolev1}
Let $u$ be a very weak solution to \eqref{3} with \eqref{4} and deficit $\epsilon>0$.
Suppose that $\widetilde{\Omega}\Subset\Omega$ is a convex open set such that $B_{\rho}(y)\subset\widetilde{\Omega}\subset B_{\alpha\rho}(y)$ for some $y\in \mathbb{R}^n$, $0<\rho\leq1$ and $\alpha>1$ and $T_1,T_2\subset(0,T)$ are two intervals. Then for $1\leq\theta\leq\inf_{z\in \widetilde{\Omega}\times (T_1\cup T_2)}p(z)(1-\epsilon)$, there holds
\[\begin{split}
\fint_{\widetilde{\Omega}\times T_1}|u-(u)_{\widetilde{\Omega}\times T_2}|^{\theta}dz
&\leq
c\,\rho^{\theta}\left(\fint_{\widetilde{\Omega}\times T_1}|Du|^{\theta}dz+\fint_{\widetilde{\Omega}\times T_2}|Du|^{\theta}dz\right)\\
&\phantom{\le\ }
+c\,\rho^{-\theta}\left(\int_{T_1\cup T_2}\fint_{\widetilde{\Omega}}(1+|Du|+|F|)^{p(\cdot)-1}dz\right)^{\theta}
\end{split}\]
where the constant $c$ depends only on $n,N,L,\gamma_1$ and $\alpha$.
\end{lemma}

\begin{lemma}\label{sobolev2}
Let $M\geq1$ be fixed. Then there exists $\rho_0=\rho_0(n,\gamma_1,L,M)$ such that the following holds:  Assume that $u$ is a very weak solution to \eqref{3} with \eqref{4} and deficit $\epsilon>0$ satisfying \eqref{6}. Suppose that on the parabolic cylinder $Q_{\rho,s}(z_0)=B_{\rho}^{(\lambda)}(x_0)\times (t_0-s,t_0+s)$ with $0<\lambda^{(p_0-2)/(2p_0)}\rho\leq\rho_0$, $0<s\leq\rho^2$, $\lambda_1\geq c_E\lambda$ and $Q_{\rho,s}(z_0)\Subset Q^{(6)}$, there holds
\begin{equation}\label{22}\fint_{Q_{\rho,s}(z_0)}\left(|Du|+|F|+1\right)^{p(\cdot)(1-\epsilon)}dz\leq \lambda^{1-\epsilon}.\end{equation}
Then for any $1\leq\theta\leq\inf_{z\in Q_{\rho,s}(z_0)}p(z)(1-\epsilon)$, we have
\begin{equation}\label{23}\fint_{Q_{\rho,s}(z_0)}|u-(u)_{Q_{\rho,s}(z_0)}|^{\theta}dz\leq c\rho^{\theta}\lambda^{\theta/2},
\end{equation}
where the constant $c$ depends only on $n$, $N$, $L$ and $\gamma_1$.
\end{lemma}

\section{Construction of the test function}\label{test function}
In this section, we will construct a suitable testing function for the weak form (\ref{5}) of the parabolic system. To this aim we fix a cylinder $Q_\rho^{(\lambda)}(z_o)$ with $0<\rho\le 1$ and $\lambda\ge 1$ and $Q_{32\rho}^{(\lambda)}(z_o)\subset\Omega_T$.
Letting $\rho_1$ and $\rho_2$ be two fixed numbers such that $\rho\leq\rho_1<\rho_2\leq 16\rho$, we set
\begin{align*}
	&Q^{(0)}:=Q_{\rho}^{(\lambda)}(z_0),\quad
	Q^{(1)}:=Q_{\rho_1}^{(\lambda)}(z_0),\quad
	Q^{(2)}:=Q_{\rho_1^+}^{(\lambda)}(z_0),\quad
	Q^{(3)}:=Q_{\rho_2^-}^{(\lambda)}(z_0),\\[5pt]
	&Q^{(4)}:=Q_{\rho_2}^{(\lambda)}(z_0),\quad
	Q^{(5)}:=Q_{16\rho}^{(\lambda)}(z_0),\quad
	Q^{(6)}:=Q_{32\rho}^{(\lambda)}(z_0),
\end{align*}
where $\rho_1^+=\rho_1+\frac{1}{3}(\rho_2-\rho_1)$ and $\rho_2^-=\rho_1+\frac{2}{3}(\rho_2-\rho_1)$. We note that
$Q^{(0)}\subset Q^{(1)}\subset Q^{(2)} \subset Q^{(3)} \subset Q^{(4)}\subset Q^{(5)}\subset Q^{(6)}$.
In the following, we will write $B^{(k)}$ for the projection of $Q^{(k)}$ in $x$ direction and $\Lambda^{(k)}$ for the projection of $Q^{(k)}$ in $t$ direction for $k\in\{0,\cdots,6\}$.
Denoting by $p_1$ and $p_2$ the minimum and maximum of $p(\cdot)$ over the cylinder $Q^{(6)}$, i.e.
$$
	p_1=\inf_{Q^{(6)}}p(\cdot)\quad\mbox{and}\quad p_2=\sup_{Q^{(6)}}p(\cdot),
$$
and taking into account that $p(\cdot)\leq 2$ and that $\omega$ is concave, we find that
\begin{equation}\label{p2-p1}
	p_2-p_1
	\leq
	\omega\left(\max\left\{64\lambda^{(p_0-2)/(2p_0)}\rho,\sqrt{(64\rho)^2}\right\}\right)
	\leq
	\omega(64\rho)
	\le
	64\omega(\rho).
\end{equation}
This shows that the oscillation of $p(z)$ on the parabolic cylinder can be determined by the length of the time interval.
Therefore, by \eqref{p2-p1}, the concavity of $\omega$ and assumption \eqref{1}, we have that
\begin{align}\label{rho-bound}
	\rho^{-(p_2-p_1)}
	\le
	\rho^{-64\omega(\rho)}
	=
	\exp\big[64\omega(\rho)\log\tfrac1\rho\big]
	\le
	e^{64L}.
\end{align}
This proves that the quantity $\rho^{-(p_2-p_1)}$ can be bounded by a universal constant. Next, we are looking for a similar result for $\lambda$.
Throughout this paper, we shall assume that
\begin{equation}\label{8}
	\lambda^{1-\varepsilon}
	\le
	\fint_{Q^{(0)}}\left(|Du|+|F|+1\right)^{p(\cdot)(1-\varepsilon)}dz
\end{equation}
holds true.
Then, writing $p_0=p(z_0)$ as usual and using the fact that $|Q^{(0)}|=c(n)\rho^{2+n}\lambda^{n(p_0-2)/(2p_0)}$ and assumption \eqref{6}, we see from \eqref{8} that
$$\lambda^{\frac{(\gamma_1-2)n}{2\gamma_1}+1-\epsilon}\leq\lambda^{\frac{(p_0-2)n}{2p_0}+1-\epsilon}\leq\frac{c}{\rho^{n+2}}\int_{Q}
\left(|Du|+|F|+1\right)^{p(z)(1-\varepsilon)}dz\leq\frac{cM}
{\rho^{n+2}},$$
since $\lambda\ge1$ and $\gamma_1> \frac{2n}{n+2}$. This leads to the following upper bound for $\lambda$:
\begin{equation}\label{bound-lambda}
	\lambda\leq\left(cM\rho^{-n-2}\right)^{\frac{4\gamma_1}{(\gamma_1-2)n+2\gamma_1}},
\end{equation}
provided $\epsilon<\frac{(\gamma_1-2)n+2\gamma_1}{4\gamma_1}$.
This together with \eqref{p2-p1} and \eqref{rho-bound} implies that for $\rho>0$ with $\rho\leq\rho_0\leq \frac{1}{64M}$, there holds
\begin{equation}\label{9}
	\lambda^{(p_2-p_1)/p_0}
	\leq
	c\, \left[\rho^{-(n+2)(p_2-p_1)} M^{p_2-p_1}\right]^{\frac{4\gamma_1}{(\gamma_1-2)n+2\gamma_1}}
	\le
	c
\end{equation}
where the constant $c$ depends on $\gamma_1$, $n$ and $L$.

On the other hand, we can restrict the radius of the parabolic cylinder so small that $p_2-1$ can be bounded by $\tfrac{p_1}{\qtilde}$ for some constant $\qtilde>1$. This can be achieved by the following argument.
We fix a constant $\qtilde$ with $1<\qtilde<2$ and choose $\rho_0$ so small that $\omega(4\rho_0)\le \frac{2}{\qtilde}-1$. Then, we have for $\rho\in(0,\rho_0]$ that
\begin{equation}\label{10}
	p_2-1
	\le
	p_1+\omega(4\rho)-1
	\le
	p_1+\omega(4\rho_0)-1
	\le
	\tfrac{p_1}{\qtilde}.
\end{equation}
The estimates \eqref{rho-bound}, \eqref{9} and \eqref{10} will be frequently used throughout the paper.

Since the solution itself is not an admissible testing function, the test function will be constructed as an Whitney extension of the solution from a good set to the whole domain $\Omega_T$.
We proceed to construct this good set $E(\lambda_1)$ as follows: For $\lambda_1\ge 1$ we denote the lower level set of the maximal function
$$
	M_{Q^{(4)}}(z)
	:=
	M\left[\left(\frac{|u-u_{Q^{(1)}}|}{\lambda^{(p_0-2)/(2p_0)}\rho}+|Du|+|F|+1\right)^{p(\cdot)/\qtilde}\chi_{Q^{(4)}}
\right](z)^{\qtilde(1-\varepsilon)}
$$
by
$$
	E(\lambda_1):=\big\{z\in Q^{(4)}:\ M_{Q^{(4)}}(z)\leq\lambda_1^{1-\varepsilon}\big\}.
$$
If $E(\lambda_1)=\emptyset$,
we have by the $L^{\qtilde(1-\epsilon)}$-boundedness of the strong maximal function that
\begin{align*}
	\lambda_1^{1-\epsilon}|Q^{(4)}|
	&\leq
	\int_{Q^{(4)}}
	M_{Q^{(4)}}(z)dz
	\leq
	c\int_{Q^{(4)}}\left(\frac{|u-u_{Q^{(1)}}|}{\lambda^{(p_0-2)/(2p_0)}\rho}+|Du|+|F|+1\right)^{p(\cdot)(1-\varepsilon)}dz\\
	&\leq
	c_E^{1-\epsilon}\,\widetilde{\lambda}^{{1-\varepsilon}}|Q^{(4)}|,
\end{align*}
where
\begin{equation}\label{lambdatilde}
	\widetildelambda
	:=
	\lambda+
	\left(\fint_{Q^{(4)}}\left(\frac{|u-u_{Q^{(1)}}|}{\lambda^{(p_0-2)/(2p_0)}\rho}\right)^{p(\cdot)(1-\varepsilon)}dz\right)
^{\frac{1}{1-\varepsilon}}
\end{equation}
and $c_E$ is a constant depend on $n$ and $\gamma_1$.
For $\lambda_1>c_E\widetildelambda$, this leads to a contradiction. Therefore, we conclude that $E(\lambda_1)$ is nonempty for $\lambda_1\ge c_E\widetildelambda$.
We also note that the set $E(\lambda_1)$ is bounded and closed. Therefore, for any fixed point $z\in Q^{(4)}\setminus E(\lambda_1)$, there exists a neighbourhood $\widetilde{Q}$ such that $\widetilde{Q}\subset Q^{(4)}\setminus E(\lambda_1)$. This motivates us to establish the following Lemma.
\begin{lemma}\label{lem-lambda1}
Let $\lambda\ge 1$, $\lambda_1\ge c_E\lambda$ and $\alpha\ge 1$, $z\in Q^{(4)}\setminus E(\lambda_1)$ and define
$$
	r_z:=d_{z}(z,E(\lambda_1))
	\quad\mbox{where}\quad
	d_{z}(z_1,z_2):=\max\Big\{\ \lambda_1^{\frac{2-p(z)}{2p(z)}}|x_1-x_2|,\sqrt{|t_1-t_2|}\ \Big\}.
$$
Then for any $z_1,z_2\in Q^{(4)}\cap \alpha Q_{r_z}^{(\lambda_1)}(z)$ we have
$\lambda_1^{|p(z_1)-p(z_2)|}
	\le
	c$ with the constant $c$ depends on $\gamma_1$, $n$, $\alpha$ and $L$.
\end{lemma}

\begin{proof}
We first observe that since $p(\cdot)\le 2$ and $\lambda,\lambda_1\ge 1$ we have
\begin{align*}
	|p(z_2)-p(z_1)|
	\le
	\omega\big(\min\{2\alpha r_z,32\rho\}\big)
	\leq
	32\max\{\alpha,1\}\,\omega\big(\min\{r_z,\rho\}\big),
\end{align*}
where we also used the concavity of $\omega$.
Since $Q^{(4)}\cap Q^{(\lambda_1)}_{r_z}(z)\subset Q^{(4)}\backslash E(\lambda_1)$, we use Chebyshev inequality and the boundedness of the strong maximal functions to obtain
\[\begin{split}
	|Q^{(4)}\cap Q^{(\lambda_1)}_{r_z}(z)|
	&\leq
	|Q^{(4)}\backslash E(\lambda_1)|
	\leq
	\frac{1}{\lambda_1^{1-\varepsilon}}\int_{Q^{(4)}}M_{Q^{(4)}}(z) \,dz\\
	&\leq
	\frac{c}{\lambda_1^{1-\varepsilon}}\int_{Q^{(4)}}
	\left(\frac{|u-u_{Q^{(1)}}|}{\lambda^{(2-p_0)/(2p_0)}\rho}+|Du|+|F|+1\right)^{p(\cdot)(1-\varepsilon)}dz
	=:
	\frac{c\, \mathcal A}{\lambda_1^{1-\varepsilon}},
\end{split}\]
with the obvious meaning of $\mathcal A$.
This implies the following upper bound for $\lambda_1$: \begin{equation}\label{16}\lambda_1^{1-\varepsilon}\leq \frac{c\,\mathcal A}{|Q^{(4)}\cap {Q}_{r_z}^{(\lambda_1)}(z)|}.\end{equation}
To estimate the lower bound for $|Q^{(4)}\cap {Q}_{r_z}^{(\lambda_1)}(z)|$, we recall that $z\in Q^{(4)}\setminus E(\lambda_1)$ implies
\[\begin{split}
	|Q^{(4)}\cap {Q}_{r_z}^{(\lambda_1)}(z)|
	\ge
	c\,\min\{\lambda_1^{\frac{n(p(z)-2)}{2p(z)}}r_z^n,\lambda^{\frac{n(p_0-2)}{2p_0}}\rho^n\}
	\min\Big\{r_z^2,\rho^2\Big\}.
\end{split}\]
Notice that $\lambda^{\frac{n(p(z)-2)}{2p(z)}}=\lambda^{\frac{n(p_0-2)}{2p_0}}\lambda^{\frac{p(z)-p_0}{p_0p(z)}}\leq c\lambda^{\frac{n(p_0-2)}{2p_0}}$, $\gamma_1\leq p(z)\leq 2$ and $\lambda_1\ge c_E\lambda$ we obtain
$$|Q^{(4)}\cap {Q}_{r_z}^{(\lambda_1)}(z)|\ge
	c\,\lambda_1^{\frac{n(\gamma_1-2)}{2\gamma_1}}\min\{r_z^{n+2},\rho^{n+2}\}.$$
Together with (\ref{16}) this shows that
\begin{equation}\label{lambda1-pre}
	\lambda_1\leq\left[c\, \mathcal A\, \max\left\{\frac{1}{r_z^{n+2}},\frac{1}{\rho^{n+2}}\right\}\right]^{\frac{2\gamma_1+n\gamma_1-2n}{4\gamma_1}}.
\end{equation}
Next, we estimate $\mathcal A$ as follows:
\[\begin{split}
	\mathcal A
	&\leq
	c\, \rho^{-2}\lambda^{2(p_0-2)/(2p_0)}
	\left[M +|Q^{(1)}|^{-\frac{p_2-p_1}{p_1}}M^{\frac{p_2}{p_1}}+|Q^{(1)}|\right]
	\\&\leq
	c\, \rho^{-2}\lambda^{2(p_0-2)/(2p_0)} M^2,
\end{split}
\]
where we have used the facts that $\lambda^{p_2-p_1}\leq c$, $|Q^{(1)}|\le c$ and $|Q^{(1)}|^{-\frac{p_2-p_1}{p_1}}\le c$ which follows from \eqref{rho-bound} and \eqref{9}.
Inserting the bound for $\mathcal A$ into \eqref{lambda1-pre} and taking into account that $\rho^{-(p_2-p_1)}\le e^{4L}$ and $M^{p_2-p_1}\leq e^L$ (see \eqref{rho-bound} and \eqref{9}) we end up with
$\lambda_1^{|p(z_1)-p(z_2)|}
	\le
	c(n,\gamma_1,\alpha,L).$
This finishes the proof of the lemma.
\end{proof}
To construct our test function, we need the following version of the Whitney decomposition theorem for non-uniformly parabolic cylinders.
\begin{lemma}\label{Whitney}
There exist Whitney-type cylinders $\{Q_i\}_{i=1}^{\infty}$ with $Q_i\equiv B_i\times\Lambda_i:= Q^{(\lambda_1)}_{r_i}(z_i)$ and $z_i\in Q^{(4)}\backslash E(\lambda_1)$,
having the following properties:
\begin{itemize}
\item[(i)]
$Q^{(4)}\backslash E(\lambda_1)=\bigcup_{i=1}^{\infty}Q^{(4)}\cap Q_i,$

\item[(ii)]
In each point of $Q^{(4)}\backslash E(\lambda_1)$ intersect at most $c(n,L,\gamma_1)$ of the cylinders $2Q_i$.

\item[(iii)]
There exists a constant $c_w=c_w(n,L,\gamma_1)$ such that for all $i\in \mathbb{N}$ there holds
$$
	c_wQ_i\subset \mathbb{R}^{n+1}\backslash E(\lambda_1)
	\ \ \ \ \ \text{and}\ \ \ \ \ \
	2c_wQ_i\cap E(\lambda_1)\neq\emptyset,
$$

\item[(iv)]
There exists a constant $c=c(n,L,\gamma_1)$ such that for any Whitney cylinders $Q_i$ and $Q_j$ with $2Q_i\cap 2Q_j\neq\emptyset$, there holds
$$
	|B_i|
	\leq
	c\,|B_j|
	\leq
	c\,|B_i|
	\quad\mbox{and}\quad
	|\Lambda_i|
	\leq
	c\,|\Lambda_j|
	\leq
	c\,|\Lambda_i|.
$$

\item[(v)]
There exists a constant $\hat c=\hat c(n,L,\gamma_1)>4$ such that $2Q_i\cap 2Q_j\neq\emptyset$ implies $2Q_i\subset \hat c Q_j$.
\end{itemize}\end{lemma}
From Lemma \ref{lem-lambda1}, the proof of Lemma \ref{Whitney} follows with
the same arguments as \cite[Lemma 4.2]{bl} and the details are left to the reader.
Subordinate to the cylinders $Q_i$, we can construct a partition of unity as stated in the following lemma.
\begin{lemma}\label{unity partition}\cite[Lemma 4.3]{bl} There exists a partition of unity $\{\psi_i\}_{i=1}^{\infty}$ on $\mathbb{R}^{n+1}\backslash E(\lambda_1)$, i.e. $\sum_{i=1}^{\infty}\psi_i\equiv1$ on $\mathbb{R}^{n+1}\backslash E(\lambda_1)$ having following properties,
\begin{equation*}
	\begin{cases}
	\psi_i\in C_0^{\infty}(2Q_i),\ \ \ \ \ 0\leq\psi_i\leq1,\ \ \ \text{and}\ \ \ \ \psi_i\geq c\ \ \ \text{on $Q_i$,}\\[5pt]
	|\partial_t\psi_i|\leq cr_i^{-2},\ \ |D\psi_i|\leq c\lambda_1^{(2-p_0^i)/(2p_0^i)}r_i^{-1},
	\end{cases}
\end{equation*}
where the constant $c$ only depends on $n,L$ and $\gamma_1$.
\end{lemma}

We now use the Lipschitz truncation method to construct the test function as desired.
For $i\in\mathbb{N}$, we define $I(i):=\{j\in\N:\supp\psi_j\cap\supp\psi_i\neq\emptyset\}$ and by $\#I(i)$ we denote the number of elements in $I(i)$. From Lemma \ref{Whitney} (ii), we know that $\#I(i)\leq c(n,L,\gamma_1)$ for any $i\in\mathbb{N}$.
Furthermore, for $i\in\N$ we define the enlarged cylinder
$\widehat{Q_i}
	:=
	\hat{c}Q_i
	\equiv
	Q_{\hat r_i}^{(\lambda_1)}(z_i),$
where $\hat{r}_i:=\hat c r_i$ and $\hat c=\hat c(n,L,\gamma_1)$ denotes the constant from Lemma \ref{Whitney} (v). Then, by Lemma \ref{Whitney} (v) we see that $\bigcup_{j\in I(i)}\mathrm{supp}\ \psi_j\subset\bigcup_{j\in I(i)}2Q_j\subset\widehat{Q_i}$.
We now define the function
$v(z)\equiv v(x,t)
	:=
	\eta(x)\zeta(t)(u-u_{Q^{(1)}}),$
where $\eta\in C_0^{\infty}(B^{(3)})$, $\zeta\in C_0^{\infty}(\Lambda^{(3)})$ are cutoff functions satisfying
\begin{equation}
	\begin{cases}\label{etazeta}
	\eta\equiv1\ \text{in}\ B^{(2)}, & 0\leq\eta\leq1,\ |D\eta|\leq c\lambda^{(2-p_0)/(2p_0)}(\rho_2-\rho_1)^{-1}\\[5pt]
	\zeta\equiv1\ \text{in}\ \Lambda^{(2)},& 0\leq\zeta\leq1,\ |\partial_t\zeta|\leq c(\rho_2^2-\rho_1^2)^{-1}.
	\end{cases}
\end{equation}
It follows that $\supp (\eta\zeta)\subset Q^{(3)}$. Then, for $Q_i$ and $\psi_i$ as in Lemmas \ref{Whitney} and \ref{unity partition} we define the test function
\begin{equation}\label{test-function}
	\widetilde {v}(z)\equiv \widetilde {v}(x,t)
	:=
	\begin{cases}
	v(z), &\mbox{for $z\in E(\lambda_1)$,}\\[5pt]
	\displaystyle\sum_{i=1}^\infty v_{Q_i\cap Q^{(4)}}\psi_i(z),\quad &\mbox{for $z\in \R^{n+1}\setminus E(\lambda_1)$.}
	\end{cases}
\end{equation}
Note that $v_{Q_i\cap Q^{(4)}}\neq0$ implies that $Q_i\cap Q^{(3)}\neq\emptyset$ and consequently $\mathrm{supp}\,\psi_i\cap Q^{(3)}\neq\emptyset$. For this reason we are mainly interested in getting estimates on such cylinders and we have to introduce some more notation. We set
\begin{equation*}
	S_1
	:=
	\left\{t\in \mathbb{R}^1:
	|t-t_0|\leq \left(\rho_1+\tfrac{1}{9}(\rho_2-\rho_1)\right)^2 \right\}
\end{equation*}
and
\begin{equation*}
	S_2
	:=
	\left\{t\in \mathbb{R}^1:
	|t-t_0|\leq \left(\rho_1+\tfrac{2}{9}(\rho_2-\rho_1)\right)^2
\right\}.
\end{equation*}
and note that $\Lambda^{(1)}\subset S_1\subset S_2\subset \Lambda^{(2)}$.
Furthermore, we need to consider the set
$$
	\Theta
	:=
	\big\{i\in\mathbb{N}:\mathrm{supp}\,\psi_i\cap S_1\neq\emptyset\big\}
$$
and we decompose the set $\Theta$ as follows:
$$
	\Theta_1
	:=
	\big\{i\in\Theta:\widehat{Q_i}\subset\mathbb{R}^n\times S_2\big\}
	\ \ \ \ \text{and}\ \ \ \ \
	\Theta_2\equiv\Theta\backslash\Theta_1.
$$
We find that if $i\in\Theta_1$ and $9\lambda_1^{(p_0^i-2)/(2p_0^i)}r_i\leq\lambda^{(p_0-2)/(2p_0)}(\rho_2-\rho_1)$ then $Q_i\subset Q^{(4)}$. While if $i\in\Theta_2$ then there holds $r_i^2\geq \hat c^{-2}s$, where $s=(\rho_2-\rho_1)\rho$.

\section{Caccioppoli type inequality}\label{Caccioppoli inequality}
In this  section, we will prove the Caccioppoli inequality for very weak solutions to \eqref{3}.
The proof of the Caccioppoli inequality will be one of the main difficulties in proving the higher integrability for very weak solutions. Since the solution multiplied by a cut-off function cannot be used as a testing function, we will use the function $\widetilde{v}$, constructed in \eqref{test-function} instead.

To simplify the notation, we denote
$$
\mu\equiv \mu(\rho,\rho_1,\rho_2):=\left(\frac{\rho}{\rho_2-\rho_1}\right)^\beta
$$
for some constant $\beta$ that only depends on $n,N$ and $L$. The precise value of $\beta$ may change from line to line.
For any fixed Whitney cylinder $Q_i=Q^{(\lambda_1)}_{r_i}(z_i)$, we write $p_0^i=p(z_i)$. For abbreviation, we write
$$\rho_\lambda=\lambda^{(p_0-2)/(2p_0)}\rho\quad\text{and}\quad r_{i,\lambda_1}=
\lambda_1^{(p_0^i-2)/(2p_0^i)}r_i.$$
We remark that $\rho_\lambda$ and $r_{i,\lambda_1}$ are the radius of $B_{\rho}^{(\lambda)}$ and $B_i$ respectively.

Let $p_1^i=\min\{p(z):z\in \widehat{Q_i}\cap Q^{(4)}\}$ and $p_2^i=\max\{p(z):z\in \widehat{Q_i}\cap Q^{(4)}\}$, we use Lemma \ref{lem-lambda1} to deduce that $\lambda_1^{p_2^i-p_1^i}\leq c$. Keeping this estimate in mind, the argument from \cite[section 6]{bl} with $\rho$ and $r_{i}$ replaced by $\rho_\lambda$ and $r_{i,\lambda_1}$, the details of which we omit, suggests the following Lemma.
\begin{lemma}\label{sobolev ineuality 3 whitney cubes} Let $\lambda_1\geq c_E\widetilde{\lambda}$ and $Q_i\subset \mathbb{R}^{n+1}$ be a parabolic cylinder of Whitney type with $Q^{(3)}\cap Q_i\neq\emptyset$ then there holds,
\begin{equation}\label{5.1}
\int_{Q^{(4)}\backslash E(\lambda_1)}|\widetilde{v}|^2dz\leq c\int_{Q^{(3)}\backslash E(\lambda_1)}|v|^2dz.\end{equation}
In the case $i\in\Theta$ there holds,
\begin{equation}\label{5.22}
\int_{B^{(4)}\times S_1}|\partial_t\widetilde{v}\cdot(\widetilde{v}-v)| dz
\leq
c\lambda_1|Q^{(4)}\backslash E(\lambda_1)|+\frac{c}{s}\int_{Q^{(4)}}|v|^{2}dz.
\end{equation}
and
\begin{equation}\begin{split}\label{5.2}\int_{ Q^{(4)}\backslash E(\lambda_1)}&(|Du|+|F|+1)^{p(z)-1}\left[\lambda^{(2-p_0)/(2p_0)}(\rho_2-\rho_1)^{-1}
|\widetilde{v}|+|D\widetilde{v}|\right]dz\\&\leq c\mu
\lambda_1|Q^{(4)}\backslash E(\lambda_1)|+c\mu s^{-1}\int_{Q^{(4)}\backslash E(\lambda_1)}|v|^2dz.\end{split}\end{equation}
Moreover, in the case $i\in\Theta_1$ there holds,
\begin{equation}\label{5.3}\fint_{Q^{(4)}\cap \widehat{Q_i}}|v-v_{Q^{(4)}\cap \widehat{Q_i}}|dz\leq c\mu\min\{r_{i,\lambda_1},\rho_\lambda\}\lambda_1^{1/p_0^i},\end{equation}
\begin{equation}\label{5.4}
\sup_{Q^{(4)}\cap 2Q_i}\left|\frac{\partial\widetilde{v}}{\partial t}\right|\leq c\mu\lambda_1^{(1-p_0^i)/p_0^i}r_{i,\lambda_1}^{-1}=c\mu\lambda_1^{\frac{2}{p_0^i}-\frac{3}{2}}r_{i}^{-1},\end{equation}
\begin{equation}\label{5.5}
\fint_{Q^{(4)}\cap Q_i}
|\widetilde{v}-v|dz
\leq
c\min\{r_{i,\lambda_1},\rho_\lambda\} \lambda_1^{1/p_0^i},
\end{equation}
and in the case $i\in\Theta_2$ there holds,
\begin{equation}\label{5.6}
\sup_{Q^{(4)}\cap 2Q_i}\left|\frac{\partial\widetilde{v}}{\partial t}\right|\leq c\mu s^{-1}\lambda_1^{\frac{1}{p_0^i}}\rho_{\lambda}.\end{equation}
Specifically, for a.e. $t\in S_1$ we have
\begin{equation}\label{5.7}\int_{B^{(4)}\backslash E_t(\lambda_1)}(|v|^2-|v-\widetilde{v}|^2)(\cdot,t)dx\geq -c\mu\lambda_1|Q^{(4)}\backslash E(\lambda_1)|-c\mu s^{-1}\int_{Q^{(4)}}|u-u_{Q^{(1)}}|^2dz.\end{equation}
The constants $c$ in the above estimates depend only on $n,N,L$ and $\gamma_1$.
\end{lemma}
This is a standard result which can be proved by the method in \cite[section 6]{bl} and no proof will be given here.

Next, we study the Lipschitz continuity of $\widetilde{v}$ on $B^{(4)}\times S_1$. This property will be essential in the proof of the  Caccioppoli inequality, since it ensures that $\widetilde{v}$ is an admissible testing function for the parabolic system. For simplicity of notation, we let $Q_4^1$ and $Q_5^2$ stand for $B^{(4)}\times S_1$ and $B^{(5)}\times S_2$ respectively. It is also necessary to check the equivalence of the two parabolic metrics $d_P$ and $d_z$ where $z\in Q^{(4)}$. For any fixed $z_1,z_1\in Q^{(6)}$, we first observe that
\begin{equation}\begin{split}\label{metricdp}d_z(z_1,z_2)
= \max\left\{\lambda_1^{(2-p(z))/[2p(z)]}|x_1-x_2|,\sqrt{|t_1-t_2|}\right\}
\geq d_P(z_1,z_2),\end{split}\end{equation}
since $p(z)\leq 2$ and $\lambda_1\geq1$. On the other hand, since $\lambda_1^{(2-p(z))/[2p(z)]}\leq \lambda_1^{(2-p_1)/(2p_1)}$ then we get
\begin{equation}\begin{split}\label{metricdz}d_z(z_1,z_2)\leq 2\lambda_1^{(2-p_1)/(2p_1)}d_P(z_1,z_2)\end{split}\end{equation}
for any $z_1,z_1\in Q^{(6)}$. Hence, $d_z$ and $d_P$ are equivalent for any $z\in Q^{(4)}$.
\begin{lemma}\label{Lipschitz extension in space direction} Let $\lambda_1\geq c_E\widetilde{\lambda}$. Then for any $z_1,z_2\in B^{(4)}\times S_1$ there exists a constant $K>0$ such that
\begin{equation}\label{lipschitz}|\tilde v(z_1)-\tilde v(z_1)|\leq K\left(|x_1-x_2|+\sqrt{|t_1-t_2|}\right).\end{equation}
where the constant $K$ depends on $\lambda$,
$\lambda_1$,
$p_1$, $p_2$, $\rho$, $\rho_1$, $\rho_2$ and $\|v\|_{L^1(Q^{(4)})}$.
\end{lemma}
\begin{proof}We use the integral characterization of Lipschitz continuous functions due to Da Prato (\cite[page 32]{BDM} or \cite[Theorem 3.1]{D}) to prove this Lemma.
For $z_w=(x_w,t_w)\in \overline{Q_4^1}$, we define
$$T_{r}(w)=\frac{1}{|Q_4^1\cap Q_r(z_w)|^{1+\frac{1}{n+2}}}\int_{Q_4^1\cap Q_r(z_w)}\big|\widetilde{v}-\widetilde{v}_{Q_4^1\cap Q_r(z_w)}\big|dz$$
where $Q_{r}(z_w)=B_{r}(x_w)\times(t_w-r^2,t_w+r^2)$. We are going to show that $T_{r}(w)$ is bounded independent of $z_w$ and $r$. To this aim we shall distinguish between the following four cases:
\begin{equation*}
\begin{cases}
2Q_r(z_w)\subset Q_5^2\backslash E(\lambda_1),\\
2Q_r(z_w)\cap E(\lambda_1)\neq \emptyset,\ 2Q_r(w)\subset Q_5^2\ \text{and}\ r<\tfrac{1}{3}(\rho_2-\rho_1),\\
2Q_r(z_w)\cap E(\lambda_1)\neq \emptyset,\ 2Q_r(w)\subset Q_5^2\ \text{and}\ r\geq\tfrac{1}{3}(\rho_2-\rho_1), \\ 2Q_r(w)\backslash Q_5^2 \neq \emptyset,\\
\end{cases}
\end{equation*}

In the first case, we observe that $|Q_4^1\cap Q_r(z_w)|\geq c_nr^{n+2}$ and this implies
\begin{equation*}\begin{split}\label{1st case}T_{r}(w)&\leq c_nr^{-1}\fint_{Q_4^1\cap Q_r(w)}\fint_{Q_4^1\cap Q_r(w)}\big|\widetilde{v}(z)-\widetilde{v}(\tilde{z})\big|dzd\tilde z\\&\leq c_n\max_{z\in Q_4^1\cap Q_r(w)}|D\widetilde{v}(z)|+c_nr\max_{z\in Q_4^1\cap Q_r(w)}|\partial_t\widetilde{v}(z)|,\end{split}\end{equation*}
since $\widetilde{v}$ is smooth on $Q_5^2\backslash E(\lambda_1)$. Now we consider $z\in Q_4^1\cap Q_r(z_w)$. Then there exists $i\in\Theta$ such that $z\in Q_i$. Since $2Q_r(z_w)\subset Q_5^2\backslash E(\lambda_1)$ we have $d_P(z,E(\lambda_1))\geq r$.
 Let $\hat z_i\in E(\lambda_1)$ be the point such that $d_{z_i}(z_i,\hat z_i)=d_{z_i}(z_i,E(\lambda_1))\leq 2\hat{c}r_i$. We now use (\ref{metricdp}) to infer that
\[\begin{split}r\leq d_P(z,\hat z_i)\leq d_P(z,z_i)+d_P(z_i,\hat z_i)\leq d_{z_i}(z,z_i)+d_{z_i}(z_i,E(\lambda_1))\leq 3\hat{c}r_i.\end{split}\]
From the above inequality and the definition of $\widetilde{v}$, we use Lemma \ref{unity partition} to conclude that
\[\begin{split}|D\widetilde{v}(z)|+r|\partial_t\widetilde{v}(z)|&\leq \sum_{j\in I(i)}|D\psi_j||v_{Q_j}-v_{Q_i}|+r\sum_{j\in I(i)}|\partial_t\psi_j||v_{Q_j}-v_{Q_i}|
\\&\leq\left(\frac{c\lambda_1^{(2-p_0^i)/(2p_0^i)}}{r_i}+\frac{3\hat{c}r_i}{r_i^2}\right)\fint_{\widehat{Q_i}\cap Q^{(4)}}|v-v_{\widehat{Q_i}\cap Q^{(4)}}|dz
\\&\leq\frac{c(\lambda_1)}{r_i}\fint_{\widehat{Q_i}\cap Q^{(4)}}|v-v_{\widehat{Q_i}\cap Q^{(4)}}|dz,\end{split}\]
Keeping the estimate in mind, we apply Lemma \ref{sobolev ineuality 3 whitney cubes} (\ref{5.3}) to find that in the case $i\in\Theta_1$,
\begin{equation}\label{1st case}T_{r}(w)\leq \frac{c(\lambda_1)}{r_i}\fint_{\widehat{Q_i}\cap Q^{(4)}}|v-v_{\widehat{Q_i}\cap Q^{(4)}}|dz\leq c\mu\lambda_1^{1/2}.\end{equation}
In the case $i\in \Theta_2$, we have
$r_i^2\geq \hat c^{-1}(\rho_2-\rho_1)^2$ and therefore
$$|\widehat{Q_i}\cap Q^{(4)}|\geq c|Q_i|\geq c\lambda_1^{n(p_0^i-2)/(2p_0^i)}r_i^{n+2}\geq A(n,\lambda_1,p_1,p_2,\rho_1,\rho_2),$$
where the constant $A$ depends on $n,\lambda_1,p_1,p_2,\rho_1$ and $\rho_2$.
It follows that
\begin{equation}\label{2nd case}T_{r}(w)\leq \frac{c(\lambda_1)}{r_i}\fint_{\widehat{Q_i}\cap Q^{(4)}}|v|dz\leq A\|v\|_{L^1(Q^{(4)})}\leq A_1(n,\lambda_1,p_1,p_2,\rho_1,\rho_2).\end{equation}
The upper bounds in (\ref{1st case}) and (\ref{2nd case}) are independent of $z_w$ and $r$ and this proves the Lemma in the first case.

We now turn our attention to the second case. Since $z\in Q_4^1\cap Q_r(z_w)$ then it is easy to check that $|Q_r(z_w)\cap Q_4^1|\geq c|Q_r(w)|$. We obtain
\[\begin{split}T_{r}(w)\leq \frac{c_n}{|Q_r(z_w)|^{1+\frac{1}{n+2}}}\int_{Q_r(z_w)\cap Q_4^1}2|\widetilde{v}-v|+|v-v_{Q_r(z_w)\cap Q_4^1}|dz=:c(2T_1+T_2),\end{split}\]
with the obvious meaning of $T_1$ and $T_2$. To estimate $T_2$, we apply the arguments in the spirit of the proof of \cite[Lemma 5.11]{B}.
Similarly to there, we construct a weight function $\hat{\eta}\in C_0^\infty(B_r(x_w)\cap B^{(4)})$ satisfying $\hat\eta\geq0$, $\int_{\R^n}\hat\eta dx=1$ and $|D\hat\eta|\leq c\max\{r^{-1-n},\rho^{-1-n}\}$. Let $v_{\hat\eta}(t)=\int_{\mathrm{R}^n}(v\hat\eta)(\cdot,t)\, dx$, we conclude that
\[\begin{split}
T_2&\leq \frac{c}{r}\fint_{Q_r(z_w)\cap Q_4^1}|v-v_{\hat\eta}| dz+\frac{c}{r}\max_{t_1,t_2\in S_1\cap(t_w-r^2,t_w+r^2)}\left|v_{\hat\eta}(t_2)-
v_{\hat\eta}(t_1)\right|\\&\leq \frac{c}{r}\min\{r,\lambda^{\frac{p_0-2}{2p_0}}\rho\}\fint_{Q_r(z_w)\cap Q_4^1}|Dv| dz
+\frac{c}{r}\max_{t_1,t_2\in S_1\cap(t_w-r^2,t_w+r^2)}\left|v_{\hat\eta}(t_2)-
v_{\hat\eta}(t_1)\right|\\&=:T_2^{(1)}+T_2^{(2)},\end{split}\]
with the obvious meaning of $T_2^{(1)}$ and $T_2^{(2)}$. In order to estimate $T_2^{(1)}$, we fix a point $z_w^\prime\in 2Q_r(z_w)\cap E(\lambda_1)$. Then we have
\begin{equation}\begin{split}\label{3rd case}T_2^{(1)}&\leq c\fint_{2Q_r(z_w)}\left(|Du|+\lambda^{(2-p_0)/(2p_0)}(\rho_2-\rho_1)^{-1}|u-u_{Q^{(1)}}|+1\right)
^{p(z)/\tilde q} dz\\&\leq c\mu M_{Q^{(4)}}(z_w^\prime)^{1/(\tilde q(1-\epsilon))}\leq c\mu\lambda_1^{1/\tilde q}.\end{split}\end{equation}
We now proceed to estimate $T_2^{(2)}$. Since $r<\tfrac{1}{3}(\rho_2-\rho_1)$, we have $ S_1\cap(t_w-r^2,t_w+r^2)\subset S_2$. Notice that $\zeta(t)\equiv 1$ on $S_2$, we have $v(x,t)=(u-u_{Q^{(1)}})\eta(x)$ whenever $t\in S_1\cap(t_w-r^2,t_w+r^2)$. Then we apply Steklov formula (\ref{11}) with $\varphi=\eta\hat\eta$, and obtain for $h>0$ and $t_1,t_2\in S_1\cap(t_w-r^2,t_w+r^2)$ that
\begin{equation*}\begin{split}
&\left|([u]_h)_{\eta\hat \eta}(t_2)-([u]_h)_{\eta\hat \eta}(t_1)\right|
\\&\leq\int_{t_1}^{t_2}\int_{B^{(3)}\cap B_r(x_w)}|\langle [A(z,Du)]_h,D(\eta\hat \eta)\rangle|+|\langle[B(z,Du)]_h,\eta\hat \eta\rangle|dxdt.
\end{split}\end{equation*}
Letting $h\downarrow 0$ and using assumption (\ref{4}) we observe that
\begin{equation*}\begin{split}
\left|v_{\hat \eta}(t_2)-v_{\hat \eta}(t_1)\right|
&\leq
c(1+\|D(\eta\hat \eta)\|_{L^\infty})\int_{Q^{(3)}\cap Q_r(w)} (1+|Du|+|F|)^{p(z)-1}\,dz.\end{split}\end{equation*}
To estimate the right hand side of the above estimate, we have
$|D(\eta\hat \eta)|
\leq c\mu \lambda^{(2-p_0^i)/p_0^i}r^{-1-n}$. Using these estimates we find that
\begin{equation*}\begin{split}
\left|v_{\hat \eta}(t_2)-v_{\hat \eta}(t_1)\right|
&\leq
c\mu \lambda^{(2-p_0^i)/p_0^i}r^{-1-n}
\int_{Q^{(3)}\cap Q_r(z_w)}\left(1+|Du|+|F|\right)^{p_2-1}dz \\
&\le
c\,\mu \lambda^{(2-p_0^i)/p_0^i}r^{-1-n}  M_{Q^{(4)}}(z_w^\prime)^{\frac{p_2-1}{(1-\epsilon)p_1}} |Q_r(z_w)|
\leq Br,\end{split}\end{equation*}
where the constant $B$ depends on $n,\lambda,\lambda_1,p_1,p_2,\rho_1$ and $\rho_2$. This implies that
\begin{equation}\label{4th case}T_2^{(2)}\leq B(n,\lambda,\lambda_1,p_1,p_2,\rho_1,\rho_2).\end{equation}
The next task is now to estimate $T_1$. Recalling that from the proof of the first case, we actually proved that for any $i\in\Theta$ with $2Q_i\cap (Q_4^1\cap Q_r(z_w))\neq\emptyset$,
$$\int_{\widehat{Q_i}\cap Q^{(4)}}|v-v_{\widehat{Q_i}\cap Q^{(4)}}|dz\leq A_1(n,\lambda_1,p_1,p_2,\rho_1,\rho_2)r_i|\widehat{Q_i}\cap Q^{(4)}|.$$
Notice that $\mathrm{supp}v\subset Q^{(3)}$, then we use the estimate above to obtain
\begin{equation*}\begin{split}T_1\leq\frac{c}{r^{n+3}}\int_{(Q_4^1\cap Q_r(z_w))\backslash E(\lambda_1)}|\widetilde{v}-v|dz&\leq \frac{c}{r^{n+3}}\sum_{i:2Q_i\cap Q_4^1\cap Q_r(z_w)\neq\emptyset}\int_{\widehat{Q_i}\cap Q^{(4)}}|v-v_{\widehat{Q_i}\cap Q^{(4)}}|dz\\&\leq \frac{cA_1}{r^{n+3}}\sum_{i:2Q_i\cap Q_4^1\cap Q_r(z_w)\neq\emptyset}r_i|\widehat{Q_i}\cap Q^{(4)}|,\end{split}\end{equation*}
where the constant $A_1$ is independent of $z_w$ and $r$. Now we proceed to estimate $T_1$ by using the geometric properties of Whitney cylinders. Let $w_1$ and $w_2$ be two points
in $2Q_r(z_w)$ satisfying $w_1\in 2Q_i\cap 2Q_r(z_w)$ and $w_2\in2Q_r(z_w)\cap E(\lambda_1)$. Then we use (\ref{metricdz}) to obtain
$$r_i\leq \tfrac{1}{\hat c}d_{z_i}(z_i,w_2)\leq \tfrac{1}{\hat c}[d_{z_i}(z_i,w_1)+d_{z_i}(w_1,w_2)]\leq \tfrac{1}{\hat c}[2r_i+2\lambda_1^{(2-p_1)/(2p_1)}d_P(w_1,w_2)]$$
From Lemma \ref{Whitney} (v), we see that $\hat c>4$. This implies that $r_i\leq c(p_1,\lambda_1)r$.
Then we conclude that there exists a constant $A_2$ independent of $z_w$ and $r$ such that for any $i\in\Theta$ with $2Q_i\cap (Q_4^1\cap Q_r(z_w))\neq\emptyset$ we have $\widehat{Q_i}\subset Q_{A_2r}(z_w)$. Then we can further estimate
\begin{equation}\label{5th case}T_1\leq \frac{cA_1}{r^{n+2}}\sum_{i\in\Theta:2Q_i\cap Q_4^1\cap Q_r(z_w)\neq\emptyset}|\widehat{Q_i}\cap Q_{A_2r}(z_w)|
\leq\frac{cA_1}{r^{n+2}}|Q_{A_2r}(z_w)|\leq cA_2^{n+2}A_1.\end{equation}
Combining the estimates (\ref{3rd case}), (\ref{4th case}) and (\ref{5th case}), we arrive at
\begin{equation}\label{6th case}T_r(w)\leq A(n,\lambda,\lambda_1,p_1,p_2,\rho_1,\rho_2)\end{equation}
and this proves the Lemma in the second case.

Finally we come to the third and fourth case. We first observe that in both cases we obtain $|Q_4^1\cap Q_r(z_w)|\geq B(n,\lambda,\lambda_1,p_1,p_2,\rho_1,\rho_2)$. Then we conclude from Lemma \ref{sobolev ineuality 3 whitney cubes} (\ref{5.1}) that
\begin{equation*}\begin{split}T_{r}(w)\leq B\int_{Q_4^1\cap Q_r(z_w)}|\widetilde{v}|dz\leq B\int_{Q^{(3)}}|v|dz+B\int_{Q^{(3)}\backslash E(\lambda_1)}|\widetilde{v}|dz\leq 2B\|v\|_{L^1(Q^{(4)})},
\end{split}\end{equation*}
which proves (\ref{6th case}) in the third and fourth case and the proof of Lemma \ref{Lipschitz extension in space direction} is complete.
\end{proof}
Now, we state our Caccioppoli type inequality as follows:
\begin{theorem}\label{Caccioppoli}Let $M\geq1$. Then there exist $\varepsilon=\varepsilon(n,N,L,\gamma_1)$ and $\rho_0=\rho_0(n,M,\gamma_1,L)$ such that the following holds: Suppose that $u$ is a very weak solution to the the parabolic system (\ref{5}) and let the assumptions of Theorem \ref{main theorem} be satisfied. Finally, assume that for some parabolic cylinder $Q:=Q_\rho^{(\lambda)}(z_0)$ with $32Q=Q_{32\rho}^{(\lambda)}(z_0)\subset\Omega_T$ with $0<32\rho\leq\rho_0$ the following intrinsic coupling holds:
\begin{equation}\label{a1}\lambda^{1-\varepsilon}\leq\fint_{Q}\left(|Du|+|F|+1\right)^{p(\cdot)(1-\varepsilon)}dz,\end{equation}
and
\begin{equation}\label{a2}\fint_{16Q}\left(|Du|+|F|+1\right)^{p(z)(1-\varepsilon)}dz\leq \lambda^{1-\varepsilon}.
\end{equation}
Then, for $\rho_1=\rho$ and $\rho_2=16\rho$ we have
\begin{equation}\label{cac1}\begin{split}
&\lambda^{1-\varepsilon}|Q^{(4)}|
+\sup_{t\in\Lambda}\int_{B}
|u-u_{Q}|^2m_{Q^{(4)}}^{-\varepsilon}(\cdot,t)dx\\ &\leq c\int_{Q^{(4)}}\left|\frac{u-u_{Q}}{\lambda^{\frac{p_0-2}{2p_0}}\rho}
\right|^{p(\cdot)(1-\varepsilon)}dz+c\lambda^{-\varepsilon}
\int_{Q^{(4)}}\left|\frac{u-u_{Q}}{\rho}\right|^2dz+c\int_{Q^{(4)}}(1+|F|)^{p(\cdot)(1-\epsilon)}dz,
\end{split}\end{equation}
where $m_{16Q}(z)=\max\{(c_E\widetilde{\lambda})^{1/(1-\varepsilon)},M_{16Q}(z)\}$ and $\widetilde{\lambda}$ is defined in \eqref{lambdatilde}.
Moreover, for $\rho\le\rho_1<\rho_2\le16\rho$, there holds
\begin{equation}\label{cac2}\begin{split}
	\lambda^{1-\varepsilon}|Q^{(4)}|&
+\sup_{t\in\Lambda^{(1)}}\int_{B^{(1)}}
|u-u_{Q^{(1)}}|^2m_{Q^{(4)}}^{-\varepsilon}(\cdot,t)dx\\ &\leq c\mu\int_{Q^{(4)}}
\left|\frac{u-u_{Q^{(1)}}}{\lambda^{\frac{p_0-2}{2p_0}}\rho_2}
\right|^{p(\cdot)(1-\varepsilon)}dz+c\mu\lambda^{-\varepsilon}
\int_{Q^{(4)}}\left|\frac{u-u_{Q^{(1)}}}{\rho_2}\right|^2dz\\&
+c\int_{Q^{(4)}}(1+|F|)^{p(\cdot)(1-\epsilon)}dz
+c_1\mu\varepsilon\lambda^{1-\varepsilon}|Q^{(4)}|.
\end{split}\end{equation}
In any cases, the constants $c$ depend only on $n$, $\nu$, $L$ and $\gamma_1$.
\end{theorem}
\begin{proof}
From Lemma \ref{Lipschitz extension in space direction}, the function $\tilde{v}(\cdot,\tau)$ is a Lipschitz function for any fixed $\tau\in S_1$. We fix $t\in\Lambda^{(1)}$.
Let $t_1\in S_1\backslash\Lambda^{(1)}$ with $t_1< t$ and $0<\delta\ll1$.
In the Steklov formula (\ref{11}) we now choose $\varphi(\cdot,\tau)=\eta(\cdot)\chi_\delta(\tau)[\tilde{v}]_h(\cdot,\tau)$ as a test function, where
\begin{equation*}
\chi_\delta(\tau)=	\begin{cases}
	0 & \text{on}\ \ (-\infty,\ t_1+h]\cup[t-h,\ +\infty)\\
	1+\frac{\tau-t_1-h-\delta}{\delta}& \text{on}\ \ [t_1+h,\ t_1+h+\delta]
	\\ 1& \text{on}\ \ [t_1+h+\delta,\ t-h-\delta]
\\ 1-\frac{\tau-t+h+\delta}{\delta}& \text{on}\ \ [t-h-\delta,\ t-h]\end{cases}
\end{equation*}
to infer that
\begin{equation}\begin{split}\label{identity}
&\int_{B^{(4)}}\partial_\tau[u]_h\cdot \eta\chi_\delta[\tilde{v}]_h(\cdot,\tau)+ \langle [A(z,Du)]_h,\chi_\delta D\left([\tilde{v}]_h\eta\right)\rangle(\cdot,\tau) dx\\&=
-\int_{B^{(4)}} [B(z,Du)]_h\cdot \eta\chi_\delta[\tilde{v}]_h (\cdot,\tau) dx\end{split}\end{equation}
for any $\tau\in S_1$. For the first term on the left hand side we compute
\begin{align*}
	\partial_\tau [u]_h \chi_\delta\cdot [\tilde v]_h
	=
	\tfrac{1}{2}
	\partial_\tau\big(|[v]_{h}|^2 - |[\tilde v - v]_{h}|^2\big)\chi_\delta +
	\partial_\tau [\tilde v]_{h}\cdot [\tilde v - v]_{h}\chi_\delta .
\end{align*}
Integrating over $B^{(4)}\times(t_1,t)$ and using \cite[Lemma 2.10]{BDM} we write
\begin{align*}
	\int_{t_1}^t &\int_{B^{(4)}}
	\partial_\tau [u]_h \eta\chi_\delta\cdot [\tilde v]_h \,dx\, d\tau
=S_1(\delta,h)+S_2(\delta,h)+S_3(\delta,h)
\end{align*}
where
\begin{align*}
S_1(\delta,h)&=	\tfrac{1}{2}\int_{t_1}^t\int_{B^{(4)}}
	\partial_\tau\Big[\big(|[v]_{h}|^2 - |[\tilde v-v]_{h}|^2\big)\chi_\delta\Big]\eta dxd\tau
\\S_2(\delta,h)&=-
	\tfrac{1}{2}\int_{t_1}^t\int_{B^{(4)}}
	\big(|[v]_{h}|^2 - |[\tilde v-v]_{h}|^2\big)\partial_\tau\chi_\delta\eta dxd\tau
\\
	S_3(\delta,h)&=
	\int_{t_1}^t\int_{B^{(4)}\backslash E_\tau(\lambda_1)}
	\big[\partial_\tau [\tilde v]_{h}\chi_\delta\big]_{-h}\cdot (\tilde v-v)\eta dxd\tau.
\end{align*}
Our next aim is to determine the limitation of $S_1(\delta,h)$, $S_2(\delta,h)$ and $S_3(\delta,h)$ as $\delta,h\downarrow 0$. We first observe that $S_1(\delta,h)=0$. To estimate $S_2(\delta,h)$, we have
\begin{equation}\begin{split}\label{E2}
S_2(\delta,h)&\to\tfrac{1}{2}\int_{B^{(4)}\times\{t\}}
	\big(|v|^2 - |\tilde v-v|^2\big)\eta dx-\tfrac{1}{2}\int_{B^{(4)}\times\{t_1\}}
	\big(|v|^2 - |\tilde v-v|^2\big)\eta dx\\&=:I-II
\end{split}\end{equation}
as $\delta,h\downarrow 0$ for a.e. $t,t_1$.
We now turn our attention to the estimate of $S_3(\delta,h)$. Define $Q_t=B^{(4)}\times(t_1,t)$, we observe that the set $Q_t\backslash E(\lambda_1)$ is open. This implies that for any $z\in Q_t\backslash E(\lambda_1)$, $\big[\partial_\tau [\tilde v]_{h}\chi_\delta\big]_{-h}\cdot (\tilde v-v)\eta(z)\to\partial_t\tilde v\chi_\delta\cdot (\tilde v-v)\eta(z)$ as $h\downarrow 0$. Furthermore, we will ensure that $$\left|\big[\partial_\tau [\tilde v]_{h}\chi_\delta\big]_{-h}\cdot (\tilde v-v)\eta(z)\right|\leq cF(z),$$ where $F(z)$ is defined by
$$F(z)=\sum_{i\in
\Theta}|v(z)-\tilde v(z)|\sup_{2Q_i\cap Q^{(4)}}|\partial_t\tilde v|\chi_{Q_i}+
\sum_{i\in\Theta_1}r_i^{-1}|v(z)-\tilde v(z)|\chi_{Q_i},$$
To this aim we define $N_h=\{i\in \mathbb{N}:h<r_i^2\}$ and decompose the term under consideration as follows
\begin{equation*}\begin{split}\left|\big[\partial_\tau [\tilde v]_{h}\chi_\delta\big]_{-h}\cdot (\tilde v-v)\eta\right|&\leq\sum_{i\in\Theta\cap N_h:Q_i\cap Q_t\neq\emptyset}\left|\big[\partial_\tau [\tilde v]_{h}\chi_\delta\big]_{-h}\cdot (\tilde v-v)\eta\right|\chi_{Q_i}\\&+\sum_{i\in\Theta\backslash N_h:Q_i\cap Q_t\neq\emptyset}\left|\big[\partial_\tau [\tilde v]_{h}\chi_\delta\big]_{-h}\cdot (\tilde v-v)\eta\right|\chi_{Q_i}\\&=:F_1(z)+F_2(z),\end{split}\end{equation*}
with the obvious meaning of $F_1(z)$ and $F_2(z)$. In the case $i\in\Theta\cap N_h$, we find that
$$\sup_{Q_i\cap Q_t\neq\emptyset}\left|\big[\partial_\tau [\tilde v]_{h}\chi_\delta\big]_{-h}\right|\leq\sup_{z\in2Q_i\cap Q^{(4)}}|\partial_t\tilde v|,$$
which implies $F_1(z)\leq F(z)$.
Next, we may assume that $h<\tfrac{1}{3}(\rho_2-\rho_1)^2$. Then for any $i\in\Theta\backslash N_h$ we have
$i\in\Theta_1$. Using the formula for the time derivative of Steklov averages and Lemma \ref{Lipschitz extension in space direction} (\ref{lipschitz}) to find for $i\in\Theta\backslash N_h$ and $z\in Q_t\cap Q_i$ that
\begin{equation*}\begin{split}|\partial_\tau [\tilde v]_{h}(z)|=\frac{|\tilde v(x,t+h)-\tilde v(x,t)|}{h}\leq \frac{K}{\sqrt{h}}\leq
Kr_i^{-1},
 \end{split}\end{equation*}
which shows that $F_2(z)\leq KF(z)$. It remains to prove that $F(z)$ is an integrable function.
We now use Lemma \ref{sobolev ineuality 3 whitney cubes} (\ref{5.1}), (\ref{5.4}), (\ref{5.5}), (\ref{5.6}) and Tonelli's theorem to get
\begin{equation*}\begin{split}\int_{Q^{(4)}}|F(z)|dz&\leq C\sum_{i\in
\Theta_2}\int_{Q_i\cap Q^{(4)}}|v(z)-\tilde v(z)|dz+
C\sum_{i\in\Theta_1}r_i^{-1}\int_{Q_i\cap Q^{(4)}}|v(z)-\tilde v(z)|dz
\\ &\leq \|v\|_{L^1(Q^{(4)})}+
C|Q^{(4)}|<+\infty,
\end{split}\end{equation*}
where the constant $C$ depends on $n,\lambda_1,\lambda,\rho_1,\rho_2,p_1$ and $p_2$, but independent of $h$. This shows that $F(z)$ is an integrable function and we are allowed to use Lebesgue's dominated convergence theorem to obtain
\begin{equation}\begin{split}\label{E3}\lim_{\delta\downarrow 0}\lim_{h\downarrow 0}S_3(\delta,h)=\int_{Q_t\backslash E(\lambda_1)}
\partial_t\tilde v\cdot (\tilde v-v)\eta dxdt=:III.\end{split}\end{equation}
We start with the estimate of $\uppercase\expandafter{\romannumeral2}$. We choose $t_1\in S_1\backslash\Lambda^{(1)}$ and $t>t_1$ such that
$$\uppercase\expandafter{\romannumeral2}=\frac{1}{|S_1\backslash\Lambda^{(1)}|}
\int_{S_1\backslash\Lambda^{(1)}}
\int_{B^{(4)}}(|v|^2-|\widetilde{v}-v|^2)(z)dz.$$
From Lemma \ref{sobolev ineuality 3 whitney cubes} (\ref{5.1}) and the fact that $|v|\leq c(u-u_{Q^{(1)}})$, we proceed to estimate $\uppercase\expandafter{\romannumeral2}$ as follows
\[\begin{split}\uppercase\expandafter{\romannumeral2}\leq \frac{c}{s}\int_{Q^{(3)}}|v|^2dz+\frac{c}{s}\int_{Q^{(3)}\backslash E(\lambda_1)}(|v|^2+|\widetilde{v}|^2)dz\leq \frac{c}{s}\int_{Q^{(3)}}|u-u_{Q^{(1)}}|^2dz.\end{split}\]
To deal with $\uppercase\expandafter{\romannumeral3}$, we use Lemma \ref{5.1} (\ref{5.22}) to obtain
\begin{equation*}
\uppercase\expandafter{\romannumeral3}
\leq\int_{B^{(4)}\times S_1}|\partial_t\widetilde{v}\cdot(\widetilde{v}-v)| dz\leq
c\lambda_1|Q^{(4)}\backslash E(\lambda_1)|+\frac{c}{s}\int_{Q^{(4)}}|v|^{2}dz.
\end{equation*}
Next, we integrating the remaining terms of (\ref{identity}) with respect to the time variable over $(t_1,t)$ and pass to the limit $ h\downarrow 0$ and $\delta\downarrow 0$. We decompose the domain of integration into the sets $Q^{(4)}\backslash E(\lambda_1)$ and $E(\lambda_1)$ to obtain
\begin{equation*}\begin{split}&\int_{t_1}^t\int_{B^{(4)}}\langle A(z,Du),D(\widetilde{v}\eta)\rangle dxdt+\int_{t_1}^t\int_{B^{(4)}}\langle B(z,Du),\widetilde{v}\eta\rangle dxdt
\\&= \int_{Q_t\cap E(\lambda_1)}\cdots dz+ \int_{Q_t\backslash E(\lambda_1)}\langle A(z,Du),D(\widetilde{v}\eta)\rangle+\langle B(z,Du),\widetilde{v}\eta\rangle dz\\&
=:\uppercase\expandafter{\romannumeral4}
+\uppercase\expandafter{\romannumeral5}.\end{split}\end{equation*}
We now use the growth condition (\ref{4}) and Lemma \ref{sobolev ineuality 3 whitney cubes} (\ref{5.2}) to conclude that
\begin{equation*}\begin{split}\uppercase\expandafter{\romannumeral5}&\leq c\int_{Q^{(4)}\backslash E(\lambda_1)}(1+|Du|+|F|)^{p(z)-1}(\lambda^{(2-p_0)/(2p_0)}(\rho_2-\rho_1)^{-1}|\widetilde{v}|+|D\widetilde{v}|)dz\\&\leq c\mu
\lambda_1|Q^{(4)}\backslash E(\lambda_1)|+cs^{-1}\int_{Q^{(3)}}|u-u_{Q^{(1)}}|^2dz,\end{split}\end{equation*}
which is bounded uniformly with respect to $t\in\Lambda^{(1)}$.
From the estimates above, we conclude that
\begin{equation}\begin{split}\label{iiiivv}\uppercase\expandafter{\romannumeral1}+
\uppercase\expandafter{\romannumeral4}=\uppercase\expandafter{\romannumeral2}-
\uppercase\expandafter{\romannumeral3}-
\uppercase\expandafter{\romannumeral5}\leq c\mu
\lambda_1|Q^{(4)}\backslash E(\lambda_1)|+cs^{-1}\int_{Q^{(3)}}|u-u_{Q^{(1)}}|^2dz\end{split}\end{equation}
for a.e. $t\in\Lambda^{(1)}$.

On the other hand we need to estimate the lower bound for (\ref{iiiivv}).
We first observe from Lemma \ref{sobolev ineuality 3 whitney cubes} (\ref{5.7}) that
\begin{equation*}\begin{split}\uppercase\expandafter{\romannumeral1}\geq -c\mu\lambda_1|Q^{(4)}\backslash E(\lambda_1)|-
\tfrac{c\mu}{s}\int_{Q^{(4)}}|u-u_{Q^{(1)}}|^2dz+\tfrac{1}{2}
\int_{E(\lambda_1)\times\{t\}}|v|^2dx\end{split}\end{equation*}
and the estimate (\ref{iiiivv}) can be rewritten as
\begin{equation}\begin{split}\label{iiv}\tfrac{1}{2}\int_{E(\lambda_1)\times\{t\}}|v|^2dx\leq -IV+c\mu\lambda_1|Q^{(4)}\backslash E(\lambda_1)|+c\mu s^{-1}\int_{Q^{(4)}}|u-u_{Q^{(1)}}|^2dz.\end{split}\end{equation}
We multiply both sides by $\lambda_1^{-1-\varepsilon}$ and integrate over $(c_E\widetilde{\lambda},+\infty)$ with respect to $\lambda_1$. Let $m_{Q^{(4)}}(z):=\max\{c_E\widetilde{\lambda},\ M_{Q^{(4)}}^{\frac{1}{1-\varepsilon}}(z)\}$ and $s_1=\rho_1^2$. Multiplying both sides by $\epsilon$, we get the following estimate,
\begin{equation*}\begin{split}\tfrac{1}{2}\int_{B^{(4)}}
&|v|^2(\cdot,t)m_{Q^{(4)}}
(\cdot,t)^{-\varepsilon}dx\\&\leq-
\int_{B^{(4)}\times(t_1,t_0+s_1)}(\langle A(z,Du),D(v\eta)\rangle+\langle B(z,Du),v\eta\rangle)m_{Q^{(4)}}(z)^{-\varepsilon}dz
\\&+c\mu\epsilon\int_{c_E\widetilde{\lambda}}^\infty\lambda_1^{-\varepsilon}\bigl|
\bigl\{z\in Q^{(4)}:\ M_{Q^{(4)}}(z)>\lambda_1^{1-\varepsilon}\bigr\}\bigr|d
\lambda_1
+\tfrac{c\mu}{s_1\widetilde\lambda^\varepsilon}
\int_{Q^{(4)}}|u-u_{Q^{(1)}}|^2dz\\&
=:-\uppercase\expandafter{\romannumeral6}
+\epsilon\uppercase\expandafter{\romannumeral7}+
\uppercase\expandafter{\romannumeral8},\end{split}
\end{equation*}
with the obvious meaning of $\uppercase\expandafter{\romannumeral6}$, $\uppercase\expandafter{\romannumeral7}$ and $\uppercase\expandafter{\romannumeral8}$. Since $\widetilde\lambda\geq\lambda$, it follows that
$VIII\leq c\mu s^{-1}\lambda^{-\varepsilon}
\int_{Q^{(4)}}|u-u_{Q^{(1)}}|^2dz$. Next, we use the assumption (\ref{a2}), Fubini theorem and the boundedness of strong maximal functions to infer that
\begin{equation*}\begin{split}\uppercase\expandafter{\romannumeral7}&\leq c\mu\int_{Q^{(4)}}\left(\left|\frac{u-u_{Q^{(1)}}}{\lambda^{(p_0-2)/(2p_0)}\rho}\right|+|Du|+|F|+1\right)^{p(z)(1-\varepsilon)}dz\\&\leq c\mu\int_{Q^{(4)}}\left|\frac{u-u_{Q^{(1)}}}{\lambda^{(p_0-2)/(2p_0)}\rho}\right|^{p(z)
(1-\varepsilon)}dz+c\mu\lambda^{1-\varepsilon}|Q^{(4)}|.\end{split}\end{equation*}
To estimate $\uppercase\expandafter{\romannumeral6}$, we note that $D(v\eta)(x,t)=\eta^2(x) Du(x,t)+v(x,t)D\eta(x)$ for any $t\in S_1$. Using the ellipticity and growth conditions (\ref{4}), we obtain
\begin{equation*}\begin{split}\uppercase\expandafter{\romannumeral6}&\geq \int_{B^{(4)}\times(t_1,t_0+s_1)}\frac{\langle A(z,Du),\eta^2Du\rangle-|\langle A(z,Du),vD\eta\rangle|-|\langle B(z,Du),v\eta\rangle|}{m_{Q^{(4)}}(z)^{\varepsilon}}dz\\&\geq \nu\int_{Q^{(1)}} |Du|^{p(z)}m_{Q^{(4)}}(z)^{-\varepsilon}dz-\int_{B^{(4)}
\times(t_1,t_0+s_1)}|F|^{p(z)}m_{Q^{(4)}}(z)^{-\varepsilon}dz
\\&\qquad-\frac{c\lambda^{\frac{2-p_0}{2p_0}}}{\rho_2-\rho_1}
\int_{B^{(4)}\times(t_1,t_0+s_1)}(1+|F|+|Du|)^{p(z)-1}|u-u_{Q^{(1)}}
|m_{Q^{(4)}}(z)^{-\varepsilon}dz
\\&:=\uppercase\expandafter{\romannumeral4}_1
-\uppercase\expandafter{\romannumeral4}_2-
\uppercase\expandafter{\romannumeral4}_3.
\end{split}\end{equation*}
with the obvious meaning of $\uppercase\expandafter{\romannumeral4}_1$, $\uppercase\expandafter{\romannumeral4}_2$ and $\uppercase\expandafter{\romannumeral4}_3$.
We first observe from the definition of $m_{Q^{(4)}}(z)$ that
$\uppercase\expandafter{\romannumeral4}_2\leq \int_{Q^{(4)}}|F|^{p(z)(1-\varepsilon)}dz.$
To estimate $\uppercase\expandafter{\romannumeral4}_1$, we introduce the set
$$E:=\left\{z\in Q^{(1)}:|Du|^{p(z)}\geq \varepsilon_1m_{Q^{(4)}}(z)\right\}$$
for some $0<\varepsilon_1<1$ to be determined later. For the estimate on $E$, we have
\begin{equation*}\begin{split}\int_E|Du|^{p(z)(1-\varepsilon)}dz\leq \varepsilon_1^{-\varepsilon}\int_E |Du|^{p(z)}m_{Q^{(4)}}(z)^{-\varepsilon}dz\leq c\varepsilon_1^{-\varepsilon}\uppercase\expandafter{\romannumeral4}_1.\end{split}\end{equation*}
For $z\in Q^{(1)}\backslash E$, we see that either $|Du|^{p(z)}\leq \varepsilon_1M_{Q^{(4)}}(z)^{\frac{1}{1-\varepsilon}}$ or $|Du|^{p(z)}\leq c_E\varepsilon_1\widetilde\lambda$. Then we use (\ref{etazeta}) and (\ref{a2}) to obtain
\begin{equation*}\begin{split}\int_{Q^{(1)}\backslash E}|Du|^{p(z)(1-\varepsilon)}dz&\leq c\varepsilon_1^{1-\varepsilon}\int_{Q^{(1)}}(M_{Q^{(4)}}(z)+
\widetilde\lambda^{1-\varepsilon})dz\\&\leq c\varepsilon_1^{1-\varepsilon}\lambda^{1-\varepsilon}|Q^{(4)}|
+c\varepsilon_1^{1-\varepsilon}\int_{Q^{(4)}}
\left|\frac{u-u_{Q^{(1)}}}{\lambda^{\frac{p_0-2}{2p_0}}\rho}\right|^{p(z)(1-\varepsilon)}dz.
\end{split}\end{equation*}
Summing these two estimates and multiply $\varepsilon_1^{\varepsilon}$ on both sides, we find that
\begin{equation*}\begin{split}\varepsilon_1^{\varepsilon}\int_{Q^{(4)}}
|Du|^{p(z)(1-\varepsilon)}dz\leq c\uppercase\expandafter{\romannumeral4}_1+c\varepsilon_1
\lambda^{1-\varepsilon}|Q^{(4)}|+c\varepsilon_1
\int_{Q^{(4)}}
\left|\frac{u-u_{Q^{(1)}}}{\lambda^\frac{p_0-2}{2p_0}\rho}\right|^{p(z)(1-\varepsilon)}dz.
\end{split}\end{equation*}
From the assumption (\ref{a1}) and we can choose $\varepsilon_1$ small enough to reabsorb the term $c\varepsilon_1
\lambda^{1-\varepsilon}|Q^{(4)}|$ to the left hand side. Then the term $\uppercase\expandafter{\romannumeral4}_1$ can be bounded from below,
\begin{equation*}\begin{split}\uppercase\expandafter{\romannumeral4}_1\geq c\lambda^{1-\varepsilon}|Q^{(4)}|-c\int_{Q^{(4)}}
\left(\left|\frac{u-u_{Q^{(1)}}}{\lambda^{\frac{p_0-2}{2p_0}}\rho}\right|+|F|+1\right)^{p(z)(1-\varepsilon)}dz.
\end{split}\end{equation*}
Now we come to the estimate of $\uppercase\expandafter{\romannumeral4}_3$. We observe from Young's inequality with $r=p(z)(1-\varepsilon)$
and $r^\prime=\frac{p(z)(1-\varepsilon)}{p(z)(1-\varepsilon)-1}$ that
\[\begin{split}\left|\frac{u-u_{Q^{(1)}}}{\lambda^{\frac{p_0-2}{2p_0}}(\rho_2-\rho_1)}\right|&
(1+|F|+|Du|)^{p(z)(1-\varepsilon)-1}\\&\leq \varepsilon_2(1+|F|+|Du|)^{p(z)(1-\varepsilon)}
+c(\varepsilon_2)\left|\frac{u-u_{Q^{(1)}}}
{\lambda^{\frac{p_0-2}{2p_0}}(\rho_2-\rho_1)}\right|^{p(z)(1-\varepsilon)}.\end{split}\]
Combining this estimate with the assumption (\ref{a2}) and the definition of $m_{Q^{(4)}}(z)$, we obtain
\[\begin{split}\uppercase\expandafter{\romannumeral4}_3&\leq c\lambda^{\frac{2-p_0}{2p_0}}(\rho_2-\rho_1)^{-1}
\int_{Q^{(4)}}(1+|F|+|Du|)^{p(z)(1-\varepsilon)-1}|u-u_{Q^{(1)}}|dz\\& \leq c_2\varepsilon_2\lambda^{1-\varepsilon}|Q^{(4)}|+
c\mu\int_{Q^{(4)}}\left|\frac{u-u_{Q^{(1)}}}{\lambda^{\frac{p_0-2}{2p_0}}\rho_2}\right|^{p(z)(1-\varepsilon)}dz.\end{split}\]
From the above estimate we arrive at
\begin{equation*}\begin{split}\lambda^{1-\varepsilon}|Q^{(4)}|&
+\sup_{t\in\Lambda^{(1)}}\int_{B^{(1)}}
|u-u_{Q^{(1)}}|^2m_{Q^{(4)}}^{-\varepsilon}(\cdot,t)dx\\ &\leq c\mu\int_{Q^{(4)}}
\left|\frac{u-u_{Q^{(1)}}}{\lambda^{\frac{p_0-2}{2p_0}}\rho_2}
\right|^{p(\cdot)(1-\varepsilon)}dz+c\mu\lambda^{-\varepsilon}
\int_{Q^{(4)}}\left|\frac{u-u_{Q^{(1)}}}{\rho_2}\right|^2dz\\&
+c\int_{Q^{(4)}}(1+|F|)^{p(\cdot)(1-\epsilon)}dz
+c_1\mu\varepsilon\lambda^{1-\varepsilon}|Q^{(4)}|.
\end{split}\end{equation*}
Moreover, in the special case $\rho_1=\rho$ and $\rho_2=16\rho$ we have $\mu\equiv \mathrm{constant}$. This allows us to choose $\varepsilon$ small enough to obtain
\begin{equation*}\begin{split}&\lambda^{1-\varepsilon}|Q^{(4)}|
+\sup_{t\in\Lambda}\int_{B}
|u-u_{Q}|^2m_{Q^{(4)}}^{-\varepsilon}(\cdot,t)dx\\ &\leq c\int_{Q^{(4)}}\left|\frac{u-u_{Q}}{\lambda^{\frac{p_0-2}{2p_0}}\rho}
\right|^{p(\cdot)(1-\varepsilon)}dz+c\lambda^{-\varepsilon}
\int_{Q^{(4)}}\left|\frac{u-u_{Q}}{\rho}\right|^2dz+c\int_{Q^{(4)}}(1+|F|)^{p(\cdot)(1-\epsilon)}dz,
\end{split}\end{equation*}
which proves the theorem.
\end{proof}
\section{Reverse-H\"older type inequality}
This section is intended to prove the reverse H\"older inequality under an additional assumption. Firstly, it is necessary to establish an estimate for the lower order terms which play a crucial role in the proof of the main result in this section.
\begin{lemma}\label{Estimates for the lower order terms} Let $M\geq 1$ be fixed. Then there exists $\rho_0=\rho_0(n,L,M)>0$ such that the following holds: Suppose that $u$ is a very weak solution to the parabolic system (\ref{5}) and satisfies the assumptions of Theorem \ref{main theorem}. Assume that for some parabolic cylinder $Q_{32\rho}^{(\lambda)}(z_0)\subset\Omega_T$ with $0<32\rho\leq\rho_0$ the following intrinsic coupling holds:
\begin{equation}\begin{split}\label{caass1}
\lambda^{1-\varepsilon}
\le
\fint_{Q_\rho^{(\lambda)}(z_0)}\left(|Du|+|F|+1\right)&^{p(z)(1-\varepsilon)}dz\end{split}\end{equation}
and
\begin{equation}\begin{split}\label{caass2}\fint_{Q_{16\rho}^{(\lambda)}(z_0)}\left(|Du|+|F|+1\right)^{p(z)(1-\varepsilon)}dz\le \lambda^{1-\varepsilon}.
\end{split}\end{equation}
Then there holds:
$$\fint_{Q_{4\rho}^{(\lambda)}(z_0)}\left|\frac{u-u_{Q_{4\rho}^{(\lambda)}(z_0)}}{4\rho}\right|^2 dz\leq c\lambda,$$
where the constant $c$ depends on $n,N,\gamma_1,\nu$ and $L$.
\end{lemma}
\begin{proof}We define the exponent
$\tilde p_1:=p_1(2-4\epsilon)/(2-\varepsilon p_1)$ as in the proof of \cite[Proposition 7.1]{bl}
and we want to apply Gagliardo-Nirenberg's inequality from Lemma \ref{Gagliardo-Nirenberg} with $(2,q,r,\theta)$ replaced by $(2,\tilde p_1,2(1-\varepsilon),\tilde p_1/2)$. This will be allowed, once we can ensure that
$2/\tilde p_1\leq 1+2(1-\varepsilon)/n$ holds true.
In order to check this condition, we introduce a function $$\Phi_{n,\gamma_1}(\epsilon)=\frac{\gamma_1(2-4\epsilon)}{2-\epsilon \gamma_1}-\frac{2n}{n+2-2\epsilon}$$ and observe that
the function $\Phi_{n,\gamma_1}(\epsilon)$ is continuous on the interval $(0,1/2)$. Notice that $\Phi_{n,\gamma_1}(0)>0$, since $\gamma_1>2n/(n+2)$. Then there exists a constant $\epsilon_0=\epsilon_0(n,\gamma_1)$ such that for any $\epsilon$ with $0<\epsilon<\epsilon_0$, there holds $\Phi_{n,\gamma_1}(\epsilon)>0$. From this we deduce that
$$\frac{\gamma_1(2-4\epsilon)}{2-\epsilon \gamma_1}\geq\frac{\gamma_1(2-4\epsilon)}{2-\epsilon \gamma_1}>\frac{2n}{n+2-2\epsilon},$$
which implies $2/\tilde p_1\leq 1+2(1-\varepsilon)/n$. We set $\rho_\lambda^{(1)}=\lambda^{(p_0-2)/(2p_0)}\rho_1$ and $\rho_\lambda^{(2)}=\lambda^{(p_0-2)/(2p_0)}\rho_2$. We see that $\rho_\lambda^{(1)}$ and $\rho_\lambda^{(2)}$ are the radius of $B^{(1)}$ and $B^{(2)}$ respectively.
Applying Lemma \ref{Gagliardo-Nirenberg} with $(2,q,r,\theta)$ replaced by $(2,\tilde p_1,2-2\varepsilon,\tilde p_1/2)$ slice-wise to $(u-u_{Q^{(1)}})(\cdot,t)$ we obtain
\begin{equation*}\begin{split}\label{Gagliardo}
\phi(\rho_1)&=
\fint_{Q^{(1)}}\left|\frac{u-u_{Q^{(1)}}}
{\rho_\lambda^{(1)}}\right|^2 dz \\
&\leq c\fint_{\Lambda^{(1)}}\fint_{B^{(1)}}
\left(\left|\frac{u-u_{Q^{(1)}}}{\rho_\lambda^{(1)}}\right|^{\tilde p_1}+|Du|^{\tilde p_1}dx\right)
\left(\fint_{B^{(1)}}\left|\frac{u-u_{Q^{(1)}}}{\rho_\lambda^{(1)}}\right|
^{2-2\varepsilon}dx\right)^{\frac{\sigma-\tilde p_1}{2(1-\varepsilon)}} dt,\end{split}\end{equation*}
where$$\phi(r)=
\fint_{Q_r^{(\lambda)}}\left|\frac{u-u_{Q_r^{(\lambda)}}}
{\lambda^{(p_0-2)/2p_0}r}\right|^2 dz.$$
Next, making use of the assumptions \eqref{caass1} and \eqref{caass2}, we can apply the Caccioppoli inequality from Theorem \ref{Caccioppoli inequality} and proceed similarly as
\cite[page 215]{bl} to obtain
\begin{equation}\begin{split}\label{phirho}\phi(\rho_1)\leq I+II+III+IV+V+VI,\end{split}\end{equation}
where
\[\begin{split}I&= c\mu\lambda^{1-\varepsilon+(\frac{2}{p_0}-\varepsilon)\frac{2-\tilde p_1}{2}},
\\II&=c\mu\lambda^{1-\varepsilon-\varepsilon\frac{2-\tilde p_1}{2}}
\left(\fint_{Q^{(4)}}\left|\frac{u-u_{Q^{(1)}}}{\rho_\lambda^{(2)}}
\right|^2dz\right)^{\frac{2-\tilde p_1}{2}},
\\III&=c\mu\lambda^{1-\varepsilon+\frac{2-p_0}{p_0}\frac{2-\tilde p_1}{2}}\left(\fint_{Q^{(4)}}
\left|\frac{u-u_{Q^{(1)}}}{\rho_\lambda^{(2)}}
\right|^2dz\right)^{\frac{(2-\tilde p_1)p_2(1-\epsilon)}{4}},
\\IV&=c\mu\lambda^{\frac{1-\varepsilon}{r^\prime}
+(\frac{2}{p_0}-\varepsilon)\frac{2-\tilde p_1}{2}}\left(\fint_{Q^{(4)}}
\left|\frac{u-u_{Q^{(1)}}}{\rho_\lambda^{(2)}}
\right|^2dz\right)^{\frac{\epsilon p_2(2-\tilde p_1)}{4}},
\\V&=
c\mu\lambda^{\frac{1-\varepsilon}{r^\prime}+\frac{2-p_0}{p_0}\frac{2-\tilde p_1}{2}}\left(\fint_{Q^{(4)}}
\left|\frac{u-u_{Q^{(1)}}}{\rho_\lambda^{(2)}}
\right|^2dz\right)^{\frac{p_2(2-\tilde p_1)}{4}},
\\VI&=c\mu\lambda^{\frac{1-\varepsilon}{r^\prime}-\varepsilon\frac{2-\tilde p_1}{2}}
\left(\fint_{Q^{(4)}}\left|\frac{u-u_{Q^{(1)}}}{\rho_\lambda^{(2)}}
\right|^2dz\right)^{\frac{(2+\epsilon p_2)(2-\tilde p_1)}{4}},\end{split}\]
and the exponents $r$ and $r^\prime$ are defined by
$$r=\frac{2(1-\epsilon)}{\epsilon(2-\tilde p_1)}\qquad\text{and}\qquad r^\prime=\frac{2(1-\epsilon)}{2(1-\epsilon)-\epsilon(2-\tilde p_1)}.$$
It can be easily seen that for $\epsilon<1/2$, we have $0<2-\tilde p_1<2$ and $r>1$. At this stage, we can repeat the arguments in \cite[page 216-218]{bl} with $(\rho,\rho_1,\rho_2)$ replaced by $(\rho_\lambda,\rho_\lambda^{(1)},\rho_\lambda^{(2)})$ to estimate $I-VI$. Since the proof of \cite[Proposition 7.1]{bl} also include the case
$p_2(1-\epsilon)\leq 2$, it is only necessary to check that the exponents in the Young's inequalities in \cite[page 216-218]{bl} are greater than one for $2n/(n+2)<p_2\leq 2$. To start with, we observe from \cite[Proposition 7.1, page 216]{bl}
that
$$1-\epsilon+\frac{(2-\epsilon p_0)(2-\tilde p_1)}{2p_0}\leq \frac{2}{p_0}+\omega(64\rho),$$
which implies
$I\leq c\mu\lambda^{2/p_0}.$
To estimate $II$, we conclude from \cite[page 216]{bl} that$$\left(1-\epsilon-\frac{\epsilon(2-\tilde p_1)}{2}\right)\frac{2}{\tilde p_1}=\frac{2}{\tilde p_1}.$$
Since $1<2/(2-\tilde p_1)$, we use Young's inequality with exponents $2/(2-\tilde p_1)$ and $2/\tilde p_1$ to obtain for any $\delta\in(0,1)$ that
\begin{equation}\begin{split}\label{Young2}II&\leq \delta\fint_{Q^{(4)}}\left|\frac{u-u_{Q^{(1)}}}{\rho_\lambda^{(2)}}
\right|^2dz+c(\delta)\mu\lambda^{\left(1-\epsilon-\frac{\epsilon(2-\tilde p_1)}{2}\right)\frac{2}{\tilde p_1}}\\&\leq\delta\fint_{Q^{(4)}}\left|\frac{u-u_{Q^{(1)}}}{\rho_\lambda^{(2)}}
\right|^2dz+c(\delta)\mu\lambda^{\frac{2}{p_0}}.\end{split}\end{equation}
Next, we consider the estimate for $III$. It is easily to check that $\frac{4}{p_2(1-\varepsilon)(2-\tilde p_1)}>1$, since $p_2\leq2$ and $\tilde p_1>0$. From the proof in \cite[page 216]{bl} and $\lambda\geq1$, we conclude that
\[\begin{split}\left(1-\varepsilon+\frac{(2-p_0)(2-\tilde p_1)}{2p_0}\right)\frac{4}{4-p_2(1-\varepsilon)(2-\tilde p_1)}\leq\frac{2}{p_0}+c(\gamma_1)\omega(64\rho).\end{split}\]
From this inequality and $\lambda^{\omega(64\rho)}\leq c$, we use Young's inequality with exponents $\frac{4}{p_2(1-\varepsilon)(2-\tilde p_1)}$ and $\frac{4}{4-p_2(1-\varepsilon)(2-\tilde p_1)}$ to find that
\begin{equation}\begin{split}\label{Young3}III&\leq \delta\fint_{Q^{(4)}}\left|\frac{u-u_{Q^{(1)}}}{\rho_\lambda^{(2)}}
\right|^2dz+c(\delta)\mu\lambda^{\left(1-\varepsilon+\frac{(2-p_0)(2-\tilde p_1)}{2p_0}\right)\frac{4}{4-p_2(1-\varepsilon)(2-\tilde p_1)}}\\&\leq\delta\fint_{Q^{(4)}}\left|\frac{u-u_{Q^{(1)}}}{\rho_\lambda^{(2)}}
\right|^2dz+c(\delta)\mu\lambda^{\frac{2}{p_0}}.\end{split}\end{equation}
We now come to the estimate of $IV$. Notice that if $\epsilon<1$ then $\frac{4}{\epsilon p_2(2-\tilde p_1)}>1$, since $p_2\leq2$ and $\tilde p_1>0$. Analysis similar to that in \cite[page 217]{bl} shows that
$$\frac{4[p_0(1-\epsilon)+(1-\epsilon p_0)(2-\tilde p_1)]}{p_0[4-\epsilon p_2(2-\tilde p_1)]}\leq\frac{2}{p_0}+c(\gamma_1)\omega(\rho).$$
At this point, we follow the same argument in \cite[page 217]{bl} and apply Young's inequality with exponents $\frac{4}{\epsilon p_2(2-\tilde p_1)}$ and  $\frac{4}{4-\epsilon p_2(2-\tilde p_1)}$ to find that
\begin{equation}\begin{split}\label{Young4}IV&\leq \delta\fint_{Q^{(4)}}\left|\frac{u-u_{Q^{(1)}}}{\rho_\lambda^{(2)}}
\right|^2dz+c(\delta)\mu\lambda^{\frac{4[p_0(1-\epsilon)+(1-\epsilon p_0)(2-\tilde p_1)]}{p_0[4-\epsilon p_2(2-\tilde p_1)]}}\\&\leq\delta\fint_{Q^{(4)}}\left|\frac{u-u_{Q^{(1)}}}{\rho_\lambda^{(2)}}
\right|^2dz+c(\delta)\mu\lambda^{\frac{2}{p_0}},\end{split}\end{equation}
since $\lambda^{\omega(\rho)}\leq c$. The estimate of the term $V$ is similar. We first observe that $\frac{4}{p_2(2-\tilde p_1)}>1$, since $p_2\leq2$ and $\tilde p_1>0$.
We also observe that the computation in \cite[page 217]{bl} actually shows that
$$\frac{2[2p_0(1-\epsilon)+(2-p_0-\epsilon p_0)(2-\tilde p_1)]}{p_0[4-p_2(2-\tilde p_1)]}\leq \frac{2}{p_0}+c(\gamma_1)\omega(\rho).$$
Applying Young's inequality with exponents $\frac{4}{p_2(2-\tilde p_1)}$ and $\frac{4}{4-p_2(2-\tilde p_1)}$ we obtain
\begin{equation}\begin{split}\label{Young5}V&\leq \delta\fint_{Q^{(4)}}\left|\frac{u-u_{Q^{(1)}}}{\rho_\lambda^{(2)}}
\right|^2dz+c(\delta)\mu\lambda^{\frac{2[2p_0(1-\epsilon)+(2-p_0-\epsilon p_0)(2-\tilde p_1)]}{p_0[4-p_2(2-\tilde p_1)]}}\\&\leq\delta\fint_{Q^{(4)}}\left|\frac{u-u_{Q^{(1)}}}{\rho_\lambda^{(2)}}
\right|^2dz+c(\delta)\mu\lambda^{\frac{2}{p_0}},\end{split}\end{equation}
since $\lambda^{\omega(\rho)}\leq c$.
We now come to the estimate for $VI$. We first choose $\epsilon<\min\{\frac{1}{4},\frac{2\gamma_1}{8-\gamma_1}\}$ and this ensures that $\frac{4}{(2+\epsilon p_2)(2-\tilde p_1)}>1$.
Furthermore, from the argument in \cite[page 218]{bl}, we deduce that for $\gamma_1<p_1\leq p_2$ there holds
$$\frac{4[1-\epsilon-\epsilon(2-\tilde p_1)]}{4-(2+\epsilon p_2)(2-\tilde p_1)}\leq \frac{2}{p_0}+c(\gamma_1)\omega(\rho).$$
At this stage, we use Young's inequality with exponent $\frac{4}{(2+\epsilon p_2)(2-\tilde p_1)}$ and $\frac{4}{4-(2+\epsilon p_2)(2-\tilde p_1)}$
to see that
\begin{equation}\begin{split}\label{Young6}VI&\leq \delta\fint_{Q^{(4)}}\left|\frac{u-u_{Q^{(1)}}}{\rho_\lambda^{(2)}}
\right|^2dz+c(\delta)\mu\lambda^{\frac{4[1-\epsilon-\epsilon(2-\tilde p_1)]}{4-(2+\epsilon p_2)(2-\tilde p_1)}}\\&\leq\delta\fint_{Q^{(4)}}\left|\frac{u-u_{Q^{(1)}}}{\rho_\lambda^{(2)}}
\right|^2dz+c(\delta)\mu\lambda^{\frac{2}{p_0}}.\end{split}\end{equation}
Plugging the estimates \eqref{Young2}-\eqref{Young6} and $I\leq c\mu\lambda^{2/p_0}$ to \eqref{phirho} and recalling that $\mu=\left(\frac{\rho}{\rho_2-\rho_1}\right)^\beta$ where $\beta$ is a constant depends only on the structural parameters, we choose $\delta=\frac{1}{10}$ to obtain
\[\begin{split}\phi(\rho_1)\leq \frac{1}{2}\phi(\rho_2)+c\left(\frac{\rho}{\rho_2-\rho_1}\right)^\beta\lambda^{\frac{\sigma}{p_0}}\end{split}\]
for any $\rho\leq \rho_1<\rho_2\leq 16\rho$. Finally we use Lemma \ref{iteration lemma} to conclude that
$$\fint_{Q_{4\rho}^{(\lambda)}(z_0)}\left|\frac{u-u_{Q_{4\rho}^{(\lambda)}(z_0)}}{\rho_\lambda}\right|^2 dz\leq c\lambda^{\frac{2}{p_0}},$$
which yields the Lemma.
\end{proof}
The following Proposition is the main result in this section.
\begin{proposition}\label{reverse holder inequality} Let $M\geq 1$ be fixed. Then there exists $\rho_0=\rho_0(n,L,M)>0$ and $\epsilon=\epsilon(n,\gamma_1)>0$ such that the following holds: Suppose that $u$ is a very weak solution to the parabolic system \eqref{3} and that the hypothesis of Theorem \ref{main theorem} are satisfied.
Finally, assume that for some parabolic cylinder $Q_{32\rho}^{(\lambda)}(z_0)\subset\Omega_T$ with $0<32\rho\leq\rho_0$ and $\lambda\ge 1$ the following intrinsic coupling holds:
\begin{equation}\label{lower}
\begin{split}
	\lambda^{1-\varepsilon}
	\le
	\fint_{Q_\rho^{(\lambda)}(z_0)}\left(|Du|+|F|+1\right)^{p(\cdot)(1-\varepsilon)}dz\end{split}\end{equation}
and
\begin{equation}\label{upper}
\begin{split}\fint_{Q_{16\rho}^{(\lambda)}(z_0)}\left(|Du|+|F|+1\right)^{p(\cdot)(1-\varepsilon)}dz\le \lambda^{1-\varepsilon}.
\end{split}\end{equation}
Then we have the following reverse-H\"older inequality:
\[\begin{split}\fint_{Q_\rho^{(\lambda)}(z_0)}|Du|^{p(\cdot)(1-\varepsilon)}dz\leq c\left[\fint_{Q_{2\rho}^{(\lambda)}(z_0)}|Du|^{\frac{p(\cdot)(1-\varepsilon)}{\bar q}}dz\right]^{\bar q}+c\fint_{Q_{2\rho}^{(\lambda)}(z_0)}(1+|F|)^{p(\cdot)(1-\varepsilon)}dz\end{split}\]
where $\bar q=\bar q(n,\gamma_1)>1$ and the constant $c$ depends only on $n,N,\gamma_1,\nu$ and $L$.
\end{proposition}
\begin{proof}
In the following we abbreviate $\rho_\lambda=\lambda^{(p_0-2)/(2p_0)}\rho$ and $\alpha Q\equiv \alpha B\times\alpha\Lambda:=Q_{\alpha\rho}^{(\lambda)}(z_0)$ for $\alpha\ge 1$.
The Caccioppoli inequality from Theorem \ref{Caccioppoli} (i.e. estimate \eqref{cac2} applied with $\rho_1=\rho$ and $\rho_2=2\rho$) implies that
\begin{equation}
\begin{split}\label{Du}
	&\fint_{Q}|Du|^{p(\cdot)(1-\varepsilon)}dz\\
	&\leq
	c\fint_{2Q}\left|\frac{u-u_{Q}}{\rho_\lambda}\right|^{p(\cdot)(1-\varepsilon)}dz+
	c\lambda^{\frac{p_0-2}{p_0}-\varepsilon}
	\fint_{2Q}\left|\frac{u-u_{Q}}{\rho_\lambda}\right|^2dz+
	c\fint_{2Q}(1+|F|)^{p(\cdot)(1-\varepsilon)}dz\\
	&\leq
	cI_{p_2(1-\varepsilon)}+c\lambda^{\frac{p_0-2}{p_0}-\varepsilon}I_{2}+
	c\fint_{2Q}(1+|F|)^{p(\cdot)(1-\varepsilon)}dz,
\end{split}\end{equation}
where we have abbreviated
\begin{equation*}
	I_{\sigma}
	:=
	\fint_{2Q}\left|\frac{u-u_{2Q}}{\rho_\lambda}\right|^\sigma dz
\end{equation*}
for $\sigma=p_2(1-\varepsilon)$ and $\sigma=2$. Let
$q_1:=2n/(n+2-2\varepsilon)$, we choose $\epsilon$ so small that $q_1<\sigma$. This can be seen as follows:

Let
$$\Psi_{n,\gamma_1}(\epsilon)=\gamma_1(1-\epsilon)-\frac{2n}{n+2-2\epsilon}.$$ From $\gamma_1>2n/(n+2)$, we observe that $\Psi_{n,\gamma_1}(\epsilon)$ is a continuous function on the interval $(0,1/2)$ and $\Psi_{n,\gamma_1}(0)>0$. Then there exists an $\epsilon_0=\epsilon_0(n,\gamma_1)>0$ such that for any $0<\epsilon<\epsilon_0$ there holds $\Psi_{n,\gamma_1}(\epsilon)>0$ and therefore  $q_1<\sigma$, since $\gamma_1\leq p_1$.

We now apply Gagliardo-Nirenberg's inequality, i.e. Lemma \ref{Gagliardo-Nirenberg} with $(\sigma,q,r,\theta)$ replaced by $(\sigma,q_1,2-2\varepsilon,q_1/\sigma)$ slice-wise to $(u-u_{2Q})(\cdot,t)$. In this way we obtain
\[\begin{split}
	I_{\sigma}
	\leq
	c\fint_{2\Lambda}\left(\fint_{2B}\left|\frac{u-u_{2Q}}{\rho_\lambda}\right|^{q_1}+|Du|^{q_1}dx\right)
	\left(\fint_{2B}\left|\frac{u-u_{2Q}}{\rho_\lambda}\right|^{2-2\varepsilon}dx\right)^{\frac{\sigma-q_1}{2-2\varepsilon}} dt.
\end{split}\]
In order to estimate the right hand side of the estimate, we use H\"older's inequality to obtain
\[\begin{split}
	\fint_{2B}\left|\frac{u-u_{2Q}}{\rho_\lambda}\right|^{2-2\varepsilon}(\cdot,t)dx
	\leq
	J^{1-\varepsilon}
\left(\fint_{2B}m_{4Q}^{1-\varepsilon}(\cdot,t)dx\right)^{\varepsilon}
\end{split}\]
for a.e. $t\in 2\Lambda$, where
$$
	J
	:=
	\sup_{t\in2\Lambda}\fint_{2B\times\{t\}}\left|\frac{u-u_{2Q}}{\rho_\lambda}\right|^{2}
m_{4Q}^{-\varepsilon}\, dx.
$$
Inserting this above and applying H\"older's inequality with exponents $r=(2-2\varepsilon)/(\varepsilon(\sigma-q_1))$ and $r^\prime=(2-2\varepsilon)/(2-2\varepsilon-\varepsilon(\sigma-q_1))$ we get
\[\begin{split}
	I_{\sigma}
	\leq
	cJ^{\frac{\sigma-q_1}{2}}
	\left(\fint_{2Q}\left|\frac{u-u_{2Q}}{\rho_\lambda}\right|^{q_1r^\prime}+|Du|^{
q_1r^\prime}dz\right)^{\frac{1}{r^\prime}}
\left(\fint_{2Q}m_{4Q}(z)^{1-\varepsilon}dz\right)^{\frac{\varepsilon(\sigma-q_1)}{2-2\varepsilon}}.
\end{split}\]
To proceed further, we apply the Caccioppoli type inequality from Theorem \ref{Caccioppoli}  (more precisely, we use \eqref{cac2} with the choice $\rho_1=2\rho$ and $\rho_2=4\rho$) and subsequently Lemma \ref{Estimates for the lower order terms} to get
$J\leq c\lambda^{\frac{2}{p_0}-\varepsilon}.$
Moreover, from Proposition \ref{Estimates for the lower order terms} we infer that
\[\begin{split}
	\fint_{2Q}m_{4Q}^{1-\varepsilon}\,dz
	&\leq
	\widetilde{\lambda}^{1-\varepsilon}+
	\fint_{2Q}M_{4Q} \,dz\leq
	c\lambda^{1-\varepsilon}.
\end{split}\]
Inserting the last two estimates above yields
\[\begin{split}
	I_{\sigma}
	&\leq
	c\lambda^{\frac{\sigma-q_1}{p_0}}
	\left(\fint_{2Q}\left|\frac{u-u_{2Q}}{\rho_\lambda}\right|^{q_2}+|Du|^{
q_2}dz\right)^{\frac{1}{r^\prime}}.
\end{split}\]
where $q_2:=q_1r^\prime$. It can be easily seen that $q_2\leq p_1(1-\epsilon)$ if we choose $\epsilon$ small enough.
We now apply Lemma \ref{sobolev1} with $(\theta,\widetilde{\Omega}\times T_1,\widetilde{\Omega}\times T_2)$ replaced by $(q_2,2Q,2Q)$ to conclude that
\[\begin{split}
	I_{\sigma}
	\leq
	c \lambda^{\frac{\sigma-q_1}{p_0}}
	\left[\left(\fint_{2Q}(1+|Du|)^{\frac{p(\cdot)q_2}{p_1}}dz\right)^{\frac{1}{r^\prime}} +
	\left(\lambda^{\frac{2-p_0}{p_0}}\fint_{2Q}(1+|Du|+|F|)^{\frac{p(\cdot)(p_2-1)}{p_2}}dz\right)^{q_1}\right] .
\end{split}\]
Now, we will find a lower bound for the exponents appearing on the right-hand side; note that there are three different exponents: $p(\cdot)q_2/p_1$ with the choices $\sigma=2$ and $\sigma=p_2(1-\epsilon)$ and $p(\cdot)(p_2-1)/p_2$. We note that in first two cases $\sigma\leq2$ and for the third one we have $p_2(1-\epsilon)/(p_2-1)\geq 2-2\epsilon$.
Our next goal is to determine the lower bound for $p_1(1-\epsilon)/q_2$.
To this aim we define a function
$$\widetilde \Phi_{n,\gamma_1}(\epsilon)=\frac{\gamma_1\left(1-\frac{n+4}{n+2}\epsilon\right)}{q_1(\epsilon)}.$$
Observe that $\widetilde \Phi_{n,\gamma_1}(\epsilon)$ is a decreasing and continuous function on $(0,1)$. Since $\gamma_1>2n/(n+2)$, we find that $\widetilde \Phi_{n,\gamma_1}(0)>1$. Then there exists an $\epsilon_0>0$ such that for any $\epsilon\in(0,\epsilon_0)$ there holds $\widetilde \Phi_{n,\gamma_1}(\epsilon)>\widetilde \Phi_{n,\gamma_1}(\epsilon_0)>1$. Next we conclude that for any $\epsilon\in(0,\epsilon_0)$, $$\frac{p_1(1-\epsilon)}{q_2}\geq \frac{\gamma_1(1-\epsilon)}{q_1r^\prime}>\widetilde \Phi_{n,\gamma_1}(\epsilon)>\widetilde \Phi_{n,\gamma_1}(\epsilon_0)>1.$$
Therefore, letting
$\bar q:=\min\left\{\widetilde \Phi_{n,\gamma_1}(\epsilon_0), 2-2\epsilon\right\}$,
we can use H\"older's inequality to obtain
\begin{equation}\begin{split}\label{sigma}
	I_{\sigma}
	&\leq
	c \lambda^{\frac{\sigma-q_1}{p_0}}
	\left(\fint_{2Q}(1+|Du|)^{\frac{p(\cdot)(1-\epsilon)}{\bar q}}dz\right)^{\frac{\bar q q_1}{p_1(1-\epsilon)}}\\&+c\lambda^{\frac{\sigma-q_1}{p_0}}\lambda^{\frac{(2-p_0)q_1}{p_0}}\left(\fint_{2Q}(1+|Du|+|F|)^{\frac{p(\cdot)(1-\epsilon)}{\bar q}}dz\right)^{\frac{\bar q q_1(p_2-1)}{p_2(1-\epsilon)}},
\end{split}\end{equation}
since $q_2:=q_1r^\prime$. Again by H\"older's inequality and the hypothesis (\ref{upper}), we obtain for the first term on the right-hand side that
\begin{equation}\label{i}
	\lambda^{\frac{\sigma-q_1}{p_0}}
	\left(\fint_{2Q}(1+|Du|)^{\frac{p(\cdot)(1-\epsilon)}{\bar q}}dz\right)^{\frac{\bar q q_1}{p_1(1-\epsilon)}}
	\le
	\lambda^{\frac{\sigma}{p_0}-1+\epsilon}
	\left(\fint_{2Q}(1+|Du|)^{\frac{p(\cdot)(1-\epsilon)}{\bar q}}dz\right)^{\bar q}.
\end{equation}
The task now is to estimate the second integral on the right-hand side. We follow the arguments similarly as \cite[page 229]{BD} to obtain
\begin{equation*}\begin{split}
	&\lambda^{\frac{\sigma-q_1}{p_0}}
	\lambda^{\frac{(2-p_0)q_1}{p_0}}
	\left(\fint_{2Q}(1+|Du|+|F|)^{\frac{p(\cdot)(1-\epsilon)}{\bar q}}dz\right)^{\frac{\bar q q_1(p_2-1)}{p_2(1-\epsilon)}} \\
	&\le
	\delta\lambda^{\frac{\sigma}{p_0}}+
	\delta^{-\frac{1}{\gamma_1-1}}\lambda^{\frac{\sigma-q_1}{p_0}}
	\left(\fint_{2Q}(1+|Du|+|F|)^{\frac{p(\cdot)(1-\epsilon)}{\bar q}}dz\right)^{\frac{\bar q q_1}{p_1(1-\epsilon)}}.
\end{split}\end{equation*}
In the same manner as (\ref{i}) we can estimate the second term of the right hand side from above and obtain
\begin{equation}\begin{split}\label{ii}
	&\lambda^{\frac{\sigma-q_1}{p_0}}
	\lambda^{\frac{(2-p_0)q_1}{p_0}}
	\left(\fint_{2Q}(1+|Du|+|F|)^{\frac{p(\cdot)(1-\epsilon)}{\bar q}}dz\right)^{\frac{\bar q q_1(p_2-1)}{p_2(1-\epsilon)}} \\
	&\le
	\delta\lambda^{\frac{\sigma}{p_0}}+\delta^{-\frac{1}{\gamma_1-1}}\lambda^{\frac{\sigma}{p_0}-1+\epsilon}
	\left(\fint_{2Q}(1+|Du|+|F|)^{\frac{p(\cdot)(1-\epsilon)}{\bar q}}dz\right)^{\bar q},
\end{split}\end{equation}
Inserting (\ref{i}) and (\ref{ii}) to (\ref{sigma}), we find that
\begin{align*}
	I_{\sigma}
	&\leq\delta\lambda^{\frac{\sigma}{p_0}}+
	c \lambda^{\frac{\sigma}{p_0}-1+\epsilon}
	\left(\fint_{2Q}(1+|Du|+|F|)^{\frac{p(\cdot)(1-\epsilon)}{\bar q}}dz\right)^{\bar q}  .
\end{align*}
Next, inserting the estimate for $I_\sigma$ with $\sigma=2$ and $\sigma=p_2(1-\epsilon)$ above to the estimate (\ref{Du}), we arrive at
\begin{equation}\begin{split}\label{Du1}
	\fint_{Q} |Du|^{p(\cdot)(1-\varepsilon)}dz
	\leq c\delta\lambda^{1-\epsilon}+
	c&\left(\fint_{2Q}(1+|Du|+|F|)^{\frac{p(\cdot)(1-\epsilon)}{\bar q}}dz\right)^{\bar q}\\&+
	c\fint_{2Q}(1+|F|)^{p(\cdot)(1-\varepsilon)}dz.
\end{split}\end{equation}
At this stage we use (\ref{lower}) in order to bound $\lambda^{1-\epsilon}$ in the following form:
$$\lambda^{1-\epsilon}\le
	c_1\fint_{Q}|Du|^{p(\cdot)(1-\varepsilon)}dz+c_2\fint_{Q}\left(|F|+1\right)^{p(\cdot)(1-\varepsilon)}dz.$$
Plugging this inequality to (\ref{Du1}) and choosing $\delta=1/(cc_1)$ we can reabsorb the integral $\frac{1}{2}\fint_{Q}|Du|^{p(\cdot)(1-\varepsilon)}dz$ into the left hand side and arrive at
\begin{align*}
	\fint_{Q} |Du|^{p(\cdot)(1-\varepsilon)}dz\leq
	c\left(\fint_{2Q}|Du|^{\frac{p(\cdot)(1-\epsilon)}{\bar q}}dz\right)^{\bar q}+
	c\fint_{2Q}(1+|F|)^{p(\cdot)(1-\varepsilon)}dz.
\end{align*}
This proves the reverse H\"older inequality.
\end{proof}
\section{Proof of the main theorem}
In this section we will prove the higher integrability of very weak solutions stated in Theorem \ref{main theorem}. The idea now, is to prove estimates for $|Du|^{p(\cdot)(1-\epsilon)}$ on certain upper level sets. The argument uses a certain stopping time argument which allows to construct a covering of the upper level sets. This method has its origin in \cite{KL1, KL2}; a slightly simplified version can be found in \cite{BD} and \cite{bl}.
Since most of the arguments are standard by now, we will only give the main ideas to the proof and refer to \cite[section 9]{bl} and \cite[\S 7]{BD} for the details.

Let $M\ge 1$ and suppose that \eqref{6} is satisfied. From now on, we consider a parabolic cylinder $Q_r\equiv Q_r(\mathfrak z_0)$ such that $Q_{2r}\Subset\Omega_T$
and define
\begin{equation}\label{def-lambda-0}
	 \lambda_0^{\frac{1}{d(p_m)}+\frac{2}{p_M}-\frac{2}{p_m}-\epsilon}
	:=
	\fint_{Q_{2r}} \big(|Du|+|F| +1\big)^{p(\cdot)(1-\epsilon)} dz
    \ge 1,
\end{equation}
where $p_M:= \sup_{Q_{2r}}p(\cdot)$, $p_m:= \inf_{Q_{2r}}p(\cdot)$ and the function $d(p)$ is defined by $d(p)=2p/(p(n+2)-2n)$. We now choose a constant $r_0(\gamma_1)>0$ such that for any $r<r_0$,
$p_M-p_m\leq \omega(2r)\leq \omega(2r_0)\leq 1/(4d(\gamma_1))$. We next choose $\epsilon<1/(12d(\gamma_1))$, then there holds:
\begin{equation}\label{epsilon}\frac{1}{d(p_m)}+\frac{2}{p_M}-\frac{2}{p_m}-\epsilon>\frac{1}{2}\left(\frac{1}{d(p_m)}+\frac{2}{p_M}-\frac{2}{p_m}\right)>0.\end{equation}
In the following argument we assume that $r<r_0$.
For fixed $r\le r_1<r_2\le 2r$ we consider the concentric parabolic cylinders
\begin{equation*}
	Q_{r}
    \subseteq
    Q_{r_1}
    \subset
    Q_{r_2}
    \subseteq
    Q_{2r}\,.
\end{equation*}
In the following we shall consider parameters $\lambda$ such that
\begin{align}\label{lambda-choice}
	\lambda > B\lambda_0,
    \qquad\mbox{where }\qquad
    B^{\frac{1}{d(p_m)}+\frac{2}{p_M}-\frac{2}{p_m}-\epsilon}
    :=
    \Big(\frac{8\chi r}{r_2-r_1}\Big)^{n+2},
\end{align}
and for $z_0\in Q_{r_1}$ we consider radii $\rho$ satisfying
\begin{equation}\label{bound-rho}
	\frac{r_2-r_1}{4\chi}
	\le
	\rho
	\le
	\frac{r_2-r_1}{2},
\end{equation}
where $\chi=\chi(n,\gamma_1)\ge 5$ denotes the constant from a version of Vitali's covering theorem \cite[Lemma 7.1]{BD} for non-uniformly parabolic cylinders.
Note that this choice ensures that $Q_\rho^{(\lambda)}(z_0)\subset Q_{r_2}$.
Recalling the definition of $\lambda_0$ we get by enlarging the domain of integration from $Q_\rho^{(\lambda)}(z_0)$ to $Q_{2r}$ the following estimate:
\begin{align*}
	\fint_{Q_\rho^{(\lambda)}(z_0)} (|Du|+|F| +1)^{p(\cdot)(1-\epsilon)} dz
	&\le
	\frac{|Q_{2r}|}{|Q_\rho^{(\lambda)}(z_0)|}
	\fint_{Q_{2r}} (|Du|+|F| +1)^{p(\cdot)(1-\epsilon)} dz \\
	&\le
	\left(\frac{8\chi r}{\lambda^{\frac{p(z_0)-2}{2p(z_0)}}(r_2-r_1)}\right)^{n+2}\lambda^{\frac{p(z_0)-2}{p(z_0)}}
    \lambda_0^{\frac{1}{d(p_m)}+\frac{2}{p_M}-\frac{2}{p_m}-\epsilon}.
\end{align*}
Since $p_m\leq p(z_0)\leq p_M$ and $\lambda\geq1$, we use \eqref{lambda-choice} to estimate the integral on the left hand side,
\begin{align*}
	\fint_{Q_\rho^{(\lambda)}(z_0)} (|Du|+|F| +1)^{p(\cdot)(1-\epsilon)} dz
	&\le
	\left(\frac{8\chi r}{\lambda^{\frac{p_m-2}{2p_m}}(r_2-r_1)}\right)^{n+2}\lambda^{\frac{p_M-2}{p_M}}
    \lambda_0^{\frac{1}{d(p_m)}+\frac{2}{p_M}-\frac{2}{p_m}-\epsilon}
    \\&=\left(\frac{8\chi r}{r_2-r_1}\right)^{n+2}\lambda^{1-\frac{1}{d(p_m)}-\frac{2}{p_M}+\frac{2}{p_m}}
    \lambda_0^{\frac{1}{d(p_m)}+\frac{2}{p_M}-\frac{2}{p_m}-\epsilon}
    \\&\leq\left(\frac{8\chi r}{r_2-r_1}\right)^{n+2}\frac{\lambda^{1-\epsilon}}{B^{\frac{1}{d(p_m)}+\frac{2}{p_M}-\frac{2}{p_m}-\epsilon}}=\lambda^{1-\epsilon}.
\end{align*}
For $\lambda$ as in \eqref{lambda-choice} we consider the upper level set
$E(\lambda,r_1):=\bigl\{z\in Q_{r_1}: |Du(z)|^{p(z)}>\lambda\bigr\}.$
In the following we show that also a reverse inequality holds true for small radii and for the Lebesgue points $z_0\in E(\lambda,r_1)$. By Lebesgue's differentiation theorem (see \cite[(7.9)]{BD}) we infer for any $z_0\in E(\lambda,r_1)$ that
\begin{align*}
	\lim_{\rho\downarrow 0} \fint_{Q_\rho^{(\lambda)}(z_0)} (|Du|+|F| +1)^{p(\cdot)(1-\epsilon)}  dz
	\ge
	|Du(z_0)|^{p_0(1-\epsilon)}
	>
    \lambda^{1-\epsilon}\,.
\end{align*}
From the preceding reasoning we conclude that the last inequality yields a radius for which the considered integral takes a value larger than $\lambda^{1-\epsilon}$, and on the other hand, the integral is smaller than $\lambda^{1-\epsilon}$ for any radius satisfying \eqref{bound-rho}. Therefore, the continuity of the integral yields the existence of a maximal radius $\rho_{z_0}$ in between, i.e.
$0<\rho_{z_0}<\frac{r_2-r_1}{4\chi}$
such that
\begin{equation}\label{stop-1}
	\fint_{Q_{\rho_{z_0}}^{(\lambda)}(z_0)} (|Du|+|F| +1)^{p(\cdot)(1-\epsilon)} dz
	=
	\lambda^{1-\epsilon}
\end{equation}
holds and
\begin{equation}\label{stop-2}
	\fint_{Q_\rho^{(\lambda)}(z_0)} (|Du|+|F| +1)^{p(\cdot)(1-\epsilon)} dz
	<
	\lambda^{1-\epsilon}
	\qquad\forall\, \rho\in (\rho_{z_0}, \tfrac{r_2-r_1}{2}]\,.
\end{equation}
At this stage we note that
$Q_{4\chi\rho_{z_0}}^{(\lambda)}(z_0)\subseteq Q_{r_2}$ and therefore by \eqref{stop-1} and \eqref{stop-2} we conclude, that the assumptions of Proposition \ref{reverse holder inequality} are fulfilled. Note here that $16\le 4\chi$ and therefore $16\rho_{z_0}\in (\rho_{z_0}, \frac{r_2-r_1}{2}]$.
We now impose the following bound on the radius $r$:
\begin{equation*}
    r\le r_1\equiv r_1(n,N,\gamma_1,\nu, L, M)\,,
\end{equation*}
where $r_1$ denotes the radius bound from Proposition \ref{reverse holder inequality}.
We are now allowed to apply Proposition \ref{reverse holder inequality}, which yields
the following Reverse-H\"{o}lder inequality:
\begin{align}\label{HI-du}
	\fint_{Q_{\rho_{z_0}}^{(\lambda)}(z_0)} |Du|^{p(\cdot)(1-\epsilon)} dz
	\le
	c \bigg( \fint_{Q_{2\rho_{z_0}}^{(\lambda)}(z_0)}
    |Du|^{\frac{p(\cdot)(1-\epsilon)}{\bar q}} dz \bigg)^{\bar q} +
	c \fint_{Q_{2\rho_{z_0}}^{(\lambda)}(z_0)} (|F|+1)^{p(\cdot)(1-\epsilon)} dz\,,
\end{align}
where $\bar q=\bar q(n,\gamma_1)>1$ and $c=(n,N,\nu,L,\gamma_1)$.

Keeping (\ref{HI-du}) in mind, we follow the arguments in \cite{bb}, \cite[section 9]{bl} and \cite[\S 7]{BD} to find a constant $r_2=r_2(n,N,\gamma_1,\nu,L,M)$ such that for any $r<r_2$ there holds the estimate
\begin{align*}
    \int_{Q_{r}} |Du|^{p(\cdot)} dz
    \le
    c\bigg(\lambda_0^{\epsilon}
    \int_{Q_{2r}} |Du|^{p(\cdot)(1-\epsilon)} dz +
	\int_{Q_{2r}} (|F|+1)^{p(\cdot)} dz\bigg)\,,
\end{align*}
for a constant $c=(n,N,\nu,L,\gamma_1)$.
Finally, passing to averages and recalling the definition of $\lambda_0$, i.e. \eqref{def-lambda-0} and the inequality \eqref{epsilon}, we deduce that
\begin{equation}\begin{split}\label{wrong-exp}
	\fint_{Q_r} |Du|^{p(\cdot)} dz
    \le
	&c \bigg(\fint_{Q_{2r}}	
    (|Du|+|F|+1)^{p(\cdot)(1-\epsilon)} dz
	 \bigg)^{1+2\epsilon\left[\frac{1}{d(p_m)}+\frac{2}{p_M}-\frac{2}{p_m}\right]^{-1}}\\& +
	c\,
	\fint_{Q_{2r}}(|F|+1)^{p(\cdot)}dz\, .
\end{split}\end{equation}
At this stage, we set $r\le r_3(M)\le 1/(4M)$
and follow the argument in \cite[page 245-247]{BD} to obtain
$$\bigg(\fint_{Q_{2r}}	
    (|Du|+|F|+1)^{p(\cdot)(1-\epsilon)} dz
	 \bigg)^{\epsilon\left[\left(\frac{1}{d(p_m)}+\frac{2}{p_M}-\frac{2}{p_m}\right)^{-1}-d(p_0)\right]}\leq c(n,L,\gamma_1),$$
where $p_0\equiv p(\mathfrak z_0)$ denotes the value of $p(\cdot)$ evaluated at the center $\mathfrak z_0$ of $Q_{2r}\equiv Q_{2r}(\mathfrak z_0)$.
This finally  completes the proof of Theorem \ref{main theorem}.
\section*{Acknowledgement}
This work was supported by the Independent Innovation Foundation of Wuhan University of Technology (Grant No. 20410685).
\bibliographystyle{abbrv}

\end{document}